\documentclass[12pt,a4paper,draft]{amsart}

\usepackage{graphicx}
\usepackage{amscd}
\usepackage[mathscr]{eucal}

\usepackage{mathrsfs}

\usepackage[all]{xy}
\usepackage{txfonts}
\usepackage{bigints}
\usepackage{multirow}
\usepackage{mathtools}
\allowdisplaybreaks

\newcommand{\adef}{\begin{defin}}
\newcommand{\zdef}{\end{defin}}
\newtheorem{defin}{Definition}
\newtheorem{theorem}{Theorem}[section]
\newtheorem{lemma}[theorem]{Lemma}
\newtheorem{proposition}[theorem]{Proposition}
\newtheorem{corollary}[theorem]{Corollary}

\def\Ave{\operatorname{Average}}
\def\dist{\operatorname{dist}}
\def\Aut{\operatorname{Aut}}

\numberwithin{equation}{section}

\numberwithin{equation}{section}
\DeclareMathOperator*{\esssup}{ess\,sup}

\newcommand{\vertiii}[1]{{\left\vert\kern-0.25ex\left\vert\kern-0.25ex\left\vert #1
    \right\vert\kern-0.25ex\right\vert\kern-0.25ex\right\vert}}

\def\block(#1,#2)#3{\multicolumn{#2}{c}{\multirow{#1}{*}{$ #3 $}}}

\DeclarePairedDelimiterX{\inp}[2]{\langle}{\rangle}{#1, #2}

\theoremstyle{definition}
\newtheorem{definition}[theorem]{Definition}

\newcommand{\N}{\mathbb{N}}
\newcommand{\St}{\mathbb{S}}
\newcommand{\U}{\mathbb{U}}
\newcommand{\D}{\mathbb{D}}
\newcommand{\T}{\mathbb{T}}
\newcommand{\V}{\mathbb{V}}
\newcommand{\To}{\longrightarrow}
\newcommand{\e}{\varepsilon}
\theoremstyle{remark}

\newcommand\PO{{\mathrm{PO}}}

\newcommand{\F}{\mathscr{F}}
\newcommand{\Fn}{\mathscr{F}^{(n)}}
\newcommand{\Fm}{\mathscr{F}^{(m)}}
\newcommand{\Fk}{\mathscr{F}^{(k)}}

\newcommand{\G}{\mathscr{G}}
\newcommand{\W}{\mathscr{W}}
\newcommand{\C}{\mathscr{C}}
\newcommand{\CX}{\mathscr{C}(X_0,X_1)}
\newcommand{\Cdual}{\mathscr{C}(X_0^*,X_1^*)}
\newcommand{\CoX}{\mathscr{C}_0(X_0,X_1)}
\newcommand{\CooX}{\mathscr{C}_{00}(X_0,X_1)}
\newcommand{\WX}{\mathscr{W}(\mathcal X)}

\newcommand{\AU}{A^\infty_\mathbb U}

\numberwithin{equation}{section}

\author{F\'elix Cabello S\'anchez}
\author{Jes\'us M.\ F.\ Castillo}
\author{Willian H.\ G.\ Corr\^ea}

\address{Instituto de Matem\'aticas Imuex, Universidad de Extremadura,
Avenida de Elvas s/n, 06011 Badajoz, Spain.}
\email{cabello@unex.es, \quad castillo@unex.es}

\address{Departamento de Matem\'atica, Instituto de Ci\^encias Matem\'aticas e de Computa\c{c}\~ao, Universidade de S\~ao Paulo, Av. Trabalhador S\~ao-carlense 400, 13566-590 S\~ao Paulo SP, Brazil}
\email{willhans@icmc.usp.br}

\thanks{The research of the first and second authors was supported by Projects MINCIN MTM2016-76958-C2-1-P and PID2019-103961GB-C21. The third author was supported by FAPESP, processes 2016/25574-8 and 2018/03765-1, and by CAPES, PDSE program, grant 88881.134107/2016-0}

\begin{document}

\title[Higher order derivatives of analytic families]{Higher order derivatives of analytic families of Banach spaces}

\maketitle


\begin{abstract} We show that the Rochberg spaces generated by complex interpolation form themselves complex interpolation scales and obtain their new interpolated spaces and associated derivations. We present our results in the context of analytic families of Banach spaces and study the problem of determining the Rochberg spaces induced by these new families.\end{abstract}

\section{Introduction}
This paper studies certain analytic families of Banach spaces that spring naturally in the context of complex interpolation of families \cite{Coifman1982}. We will work in the context of \emph{admissible} spaces $\mathscr F$ of analytic functions (Definition \ref{def:ad}) over a complex domain $\U$ as formalized by Kalton and Montgomery-Smith \cite{kalt-mon}. Starting with such an $\F$ we will consider, for $n\in \N$ and $z \in \mathbb U$, the spaces introduced by Rochberg \cite{rochpac} and formed by the arrays of the truncated sequence of the Taylor coefficients of the elements of $\F$, namely $$\Fn_{z} = \left\{\left (\frac{f^{(n-1)}(z)}{(n-1)!}, \dots, f'(z), f(z)\right): f\in \mathscr F \right\}$$
endowed with the natural quotient norm. The space $\F_z=\F^{(1)}_z$ of arrays of length one (the values of the functions of $\F$ at $z$) correspond, in the suitable context, to classical interpolation spaces, while arrays of length two (the pair formed by the values of the derivative of the functions and the values of the functions at $z$) constitute the so-called first derived space and correspond, in the suitable context, to twisted sums of the spaces $\F_z$.

Admissible spaces $\F$ emerge from complex interpolation schemas in different ways. If one has an interpolation couple $(X_0, X_1)$ and works on the complex unit strip then $\F$ could be the classical Calder\'on space $\mathscr C(X_0,X_1)$ associated to the pair. If one has a suitable family $\mathcal X= \{X_u: u\in \partial \U\}$ of Banach spaces then the complex interpolation method for families \cite{Coifman1982} can be applied to generate the space $\F$. In general, complex interpolation applied to a family $\mathcal X$ of Banach spaces on the boundary of $\U$ generates what is called an analytic family $\{X_z: z \in \U\}$ of Banach spaces on $\U$, and these are the {\em first} Rochberg spaces $X_z= \F_{z}$ for the suitably obtained admissible space $\F$. Then, one can \emph{also} form all subsequent families of Rochberg spaces $\Fn_{z}$ for $n>1$ and $z\in \U$. With this setting, our paper orbits around two axis.

The first one is the fact, implicit in Rochberg \cite{rochpac} and made explicit in \cite{cck}, that
Rochberg spaces arrange into exact sequences
\begin{equation}\label{diffprocess}
\begin{CD}
0@>>> \F_z^{(n)} @>>> \F_z^{(n+k)} @>>> \F_z^{(k)} @>>> 0
\end{CD}
\end{equation}
Rochberg \cite{rochpac} also observed that these sequences can be constructed by means of certain ``unbounded nonlinear operators $\Omega$" that we will call \emph{differentials}. We are interested in identfying the differential  $\Omega_z^{k,n}$ that generates the sequence (\ref{diffprocess}) and in using those differentials to derive information about the Rochberg spaces. The crucial example to see these ideas in action is that of the interpolation couple $(\ell_\infty, \ell_1)$, treated in Section \ref{sec:corner}. In classical Banach space theory, the middle space $B$ in an exact sequence $0\to A \to B \to C \to 0$ is usually called a twisted sum of $A$ and $C$, and correspond \cite{kaltpeck,kaltmem,kaltdiff} to a certain type of nonlinear maps called quasilinear maps. Thus, the existence of diagram $(\ref{diffprocess})$ connects the theory of Rochberg families with the theory of twisted sums of Banach spaces and Rochberg's ``unbounded nonlinear operators" with quasilinear maps.

The second axis of the paper is the connection between ``analytic families of Banach spaces" and complex interpolation. At this point, observe that the distinction between what occurs at the border and the interior of the domain $\U$ is fundamental: Rochberg derived spaces only exist on the interior. Thus, given $n>1$, it is not granted the existence of an admissible space $\mathscr T$ so that
$\mathscr T_z^{(1)} = \Fn_{z}$ for every $z\in \U$. Due to this obstruction we introduce the more general notion of \emph{acceptable} space of analytic functions (Definition \ref{def:ac}) and prove in Theorem \ref{th:generalU} one of our main results: given an acceptable space $\F$ of analytic functions on $\U$ and $n>1$ there exists an acceptable space of analytic functions $\mathscr T$ so that $\mathscr T^{(1)}_z = \Fn_z$ with equivalence of norms. Proved that, Rochberg spaces form themselves new ``acceptable families" and thus they are bound to form interpolation families. This is interesting in itself even for practical reasons since, according to Kalton and Montgomery-Smith \cite[p.1151]{kalt-mon} \emph{One of the drawbacks of the complex method is that in general it seems relatively difficult to calculate complex interpolation spaces. There is one exception to this rule, which is the case when one has a pair of Banach lattices}. Rochberg spaces are not Banach lattices and yet we can calculate the spaces obtained by complex interpolation between them.

We are thus ready to describe the organization of the paper. In Section 2, \emph{Spaces of analytic functions and complex interpolation} we recall the definition of admissible space of analytic functions, its connection with the complex interpolation method for families and introduce the notion of acceptable space. The definition of acceptable space requires the using of a Fr\'echet algebra of analytic functions, whose construction is presented in the Appendix. In Section 3, \emph{Rochberg spaces and their entwining exact sequences} we do exactly as the title says; the results can be considered a reformulation of \cite{cck} in the context of this paper. Section 4, \emph{The cornerstone example} presents a detailed study of all higher order Rochberg spaces $\mathscr Z_n$ generated by the interpolation couple $(\ell_\infty, \ell_1)$ at $z=1/2$. The first derived space is $\mathscr Z_1= \ell_2$ and the second $\mathscr Z_2=Z_2$ is the celebrated Kalton-Peck twisted Hilbert space \cite{kaltpeck}. We will obtain precise estimates for the finite dimensional type $2$ constants of all $\mathscr Z_n$ from which we deduce, for instance, that $\mathscr Z_{n+k}$ cannot isomorphic to $\mathscr Z_n\oplus \mathscr Z_k$. In Section 5, \emph{Duality issues}, we introduce, given an admissible space $\F$, a kind of admissible space $\F^\bigstar$ so that $(\F^\bigstar)^{(n)}_z$ can be interpreted as $(\F^{(n)}_z)^*$ maintaining the entwining exact sequences between the spaces. To some extent, this section is the natural extension of duality results from Kalton and Peck \cite{kaltpeck}, Rochberg \cite{rochpac} and Cwikel \cite{Cw2014}. Section 6, \emph{Analytic families of Rochberg spaces and interpolation} is the central section of the paper, where acceptable families enter the game with a substantive role. We show that if $\F$ is an acceptable space then, for each $n\geq 2$, there is an acceptable space $\mathscr T$ such that $\mathscr T_z = \F^{(n)}_z$, even if, as it is shown in the next section, $\mathscr T^{(2)}_z$ can be different from $\F^{(2n)}_z$. As we mentioned above, this is interesting in itself due to the difficulty to calculate complex interpolation spaces other than Banach lattices. Since an acceptable space of analytic functions depends on the complex domain $\U$ on which it is based, it is necessary for technical reasons to move between $\mathbb U$ and the unit disc $\mathbb D$. In particular, the difference between working on the unit strip (classical interpolation with two spaces), on the unit disc (classical interpolation for families) and on a general domain conformally equivalent to them has to be considered. We present a preparatory version on the unit strip (Proposition \ref{prop:couples}) and a general result on the unit disc (Proposition \ref{prop:Fnaccep}). Then, after a Chain and a Leibniz rule, useful to translate results from the disc to  other domains, we state and prove our main result (Theorem \ref{th:generalU}). In Section 7, \emph{Derivation of Rochberg families} we use a bit of homology to 
describe the intertwining exact sequences of Rochberg spaces in ``low dimensions'' and the way in which differentials are interlaced.

In Section 8, \emph{Applications} we solve a few problems in the literature. One of the conclusions is that the results in \cite{ccfg} are the best one can get ... while one is just considering the first Rochberg spaces. It is necessary a close inspection of higher Rochberg spaces to provide the complete panorama and answer some open problems in the literature. Finally, the Appendix, \emph{A Fr\'echet algebra of analytic functions} displays the construction of the Fr\'echet algebra of analytic functions required to sustain the notion of acceptable space.

\subsection{Notation} Domains of the complex plane are displayed in ``blackboard'' fonts: $\mathbb C$ is the complex plane, $\D$ is the unit disk and $\mathbb S$ is the unit strip. The border of a domain $\mathbb U$ will be called $\partial \mathbb U$, even if the unit circle will be always $\mathbb T$. Spaces of vector-valued analytic functions are displayed in ``mathscript'' fonts: $\mathscr F, \mathscr G$ and so on. Spaces and algebras of complex valued analytic functions follow a standard notation: $H_p, N^+, A, A^\infty, W^+$ etc. The superscript $\:^{(n)}$ is always related to derivatives while $X^{n}$ denotes the product of $n$ copies of $X$. We use the following notation for lists of Taylor coefficients. If $A$ is an ordered subset of the nonnegative integers and $f$ is analytic in a neighbourhood of $z\in\mathbb C$, then
$$\tau_A(f)=\left( \frac{f^{(n)}}{n!} \right)_{n\in A}\quad\quad\text{and}\quad\quad
\delta_z\tau_A(f)=\tau_A(f)(z)=
\left( \frac{f^{(n)}(z)}{n!} \right)_{n\in A}.
$$
In particular,
$$
\tau_n(f)=\frac{f^{(n)}}{n!}, \quad \tau_{[n,0]}(f)=\left( \frac{f^{(n)}}{n!},\dots,f \right), \quad \tau_{(n,0]}(f)=\left( \frac{f^{(n-1)}}{(n-1)!},\dots,f \right).
$$
Given $B$ a (commutative, unital) topological algebra  and a Banach space $X$, we say that $X$ is a $B$-module if there is a jointly continuous outer product $B\times X\To X$ satisfying the usual algebraical requirements. Note that in this case, for each fixed $a\in B$, the map $x\mapsto ax$ is a bounded operator on $X$ whose norm will be denoted by $\|a\|_{L(X)}$ if necessary. Also note that $B$ need not to be normed: actually the Fr\'echet algebras $A^\infty_\U$ introduced in the Appendix play a role in this paper.

\section{Spaces of analytic functions and analytic families}\label{sec:analytic-family}

This Section introduces the spaces of analytic functions that we shall use along the paper.
First, we recall the standard notion  of an admissible space of analytic functions taken from Kalton and Montgomery-Smith \cite{kalt-mon}:

\begin{definition}\label{def:ad} Let $\U$ be an open set of $\mathbb C$
conformally equivalent to the disc $\mathbb D=\{z\in\mathbb C: |z|<1\}$ and let $\Sigma$ be a complex
Banach space. A Banach space $\mathscr F$ of analytic functions
$f:\U\To \Sigma$ is said to be admissible provided:
\begin{itemize}
\item[(a)] For each $z\in \U$, the evaluation map
$\delta_z:\mathscr F\To \Sigma$ is bounded.
\item[(b)] If
$\varphi:\U\To\mathbb D$ is a conformal equivalence and $f:\U\To \Sigma$ is analytic, then
$f\in\mathscr F$ if and only if $\varphi f\in\mathscr F$ and
$\|\varphi f\|_{\mathscr F}= \|f\|_{\mathscr F}$.
\end{itemize}
\end{definition}
Condition (b) is basically a boundary condition and implies that $\F$ is isometric to the subspace of functions vanishing at a given $u\in\U$, the isometry being given by multiplication by a conformal map $\varphi:\U\to\D$ such that $\varphi(u)=0$.\medskip

It turns out that admissibility is a too rigid notion for our present purposes and so we need to introduce a weak version that we have called
(for which we apologize in advance) \emph{acceptable spaces}. This notion requires using the algebras $A^\infty_\U$, whose definition
and properties can be found in the Appendix.

\begin{definition}\label{def:ac}Let $\U$ and $\Sigma$ be as before. An acceptable space is a Banach space of analytic functions $f:\U\To \Sigma$ having the following properties:
\begin{itemize}
\item[(a)] The evaluation maps $\delta_z:\mathscr F\To \Sigma$ are continuous.
\item[(b)] $\mathscr F$ is a module over the algebra $A^\infty_\U$ under pointwise multiplication, that is, the pointwise product $A^\infty_\U\times\mathscr F\longrightarrow \mathscr F$ is jointly continuous.
\item [(c)] For each conformal mapping $\varphi:\U\To \mathbb D$ there is a constant $K[\varphi]$ such that, if $f:\mathbb U\to \Sigma$ is analytic then $\varphi f\in\mathscr F$ if and only if $f\in\mathscr F$ and then $K[\varphi]^{-1}\|\varphi f\|_\mathscr F\leq \|f\|_\mathscr F\leq K[\varphi]\|\varphi f\|_\mathscr F$.
\end{itemize}
\end{definition}


\begin{lemma}
Every admissible space of analytic functions is acceptable.
\end{lemma}

\begin{proof}
It suffices to check (b). Assume $\F$ is admissible on $\U$ and let us fix a conformal map $\psi:\U\To\D$. Then, for $f\in\F$ we have $\psi f\in\F$, with $\|\psi f\|_\F=\|f\|_\F$. Thus,
if $(c_n)_{n\geq 0}$ is absolutely summable and $g(u)=\sum_{n\geq 0}c_n\psi(u)^n$, then $g f\in\F$, with $\|gf\|_\F\leq \left(\sum_{n\geq 0}|c_n|\right)\|f\|_\F$. This implies that $\F$ is a ``contractive'' module over $\psi^*[W^+]$ under the pointwise multiplication. The definition of $\psi^*[W^+]$ is given in the Appendix. But
$\psi^*[W^+]$ contains $A^\infty_\U$ with continuous inclusion (see Lemma~\ref{lem:AW}); therefore
$\F$ is an $A^\infty_\U$-module as well. \end{proof}

\subsection{Calder\'on spaces}

The simplest examples of admissible spaces are the Calder\'on spaces associated to Banach interpolation couples.
Interpolation for couples is usually done in the unit strip $\St=\{0<\Re(z)<1\}$.
In this paper we need to be careful with the spatial variable which is used to differentiate functions and thus the size of the strip where the spaces are placed needs to be taken in consideration; see Section~\ref{sec:chainrule}. So, given real numbers $a<b$ we put $\St_{a,b}=\{a<\Re(z)<b\}.$
Now suppose that $(X_a,X_b)$ is a Banach couple: this just means that  $X_a$ and $X_b$ are Banach spaces linear and continuously embedded into a third Banach space $\Sigma$.\medskip

The Calder\'on space with simplest definition is $\C=\mathscr C(X_a,X_b)$, which consists of those bounded analytic functions $f:\St_{a,b}\To \Sigma$ that extend continuously to the closure of $\St_{a,b}$ and, denoting again by $f$ the extension, satisfy the boundary condition that for $j=a,b$ the restriction $t\in\mathbb R\longmapsto f(j+it)\in X_j$ is continuous and bounded. The norm of the space $\mathscr C$ is defined by $\|f\|_{\mathscr C}= \sup\{\|f(j + it)\|_{X_j} : t \in \mathbb{R}, j = a , b\}.$ A useful variant is the space
$$\mathscr C_0=\{f\in\mathscr C: \|f(j + it)\|_{X_j} \to 0 \text{ as } |t|\to\infty, j=a,b\},
$$ which is a closed subspace of $\mathscr C$. It is easy to prove that, if $f\in\mathscr C$, then for every $z\in\St_{a,b}$ the function $w\longmapsto e^{(w-z)^2}f(w)$ belongs to $\mathscr C_0$. Moreover, if $\Delta$ is \emph{any} dense subset of $X_a\cap X_b$, then the functions of the form
\begin{equation}\label{petunin}
f(z)= e^{\delta z^2}\sum_{1\leq i\leq k} e^{\lambda_k z}x_k\quad\quad(x_k\in\Delta, \lambda_k,\delta\in\mathbb R, \delta>0)
\end{equation}
are dense in $\mathscr C_0$; see \cite[Chapter~IV, Theorem 1.1, p. 220]{petunin}
or \cite[Lemma 4.2.3]{BL}. We shall denote the space of such functions as $\mathscr C_{00}$.

\subsection{Interpolation families}\label{sec:int-fam} The basic source of admissible spaces is the complex interpolation method for families. The method we present here, which is that of \cite{Correa2018}, is a slight modification of the method from \cite{Coifman1982}.\medskip

Let $\U$ be a domain of the complex plane conformally equivalent to the disc and let $\varphi:\D\To\U$ be a fixed conformal map. Conformal maps $\varphi$ belong to the Smirnov class $N^+$ \cite{duren} and so they have nontangencial limits for almost every $z\in\mathbb T$. Let us assume from now on that $\varphi$ extends to a surjective continuous function $\tilde{\D}\To \overline{\U}$ (there is no need to relabel), where  $\tilde{\D}$ is a subset of the closed disc which contains $\D$ together with almost every point of $\T$.
In particular $\varphi$ maps $\T\cap \tilde{\D}$ onto $\partial\U$ (up to a null set). Note that this is actually a property of the domain, and domains such as spiral of infinite turns approaching the unit circle lacks it. When $\U=\mathbb S$ one can use the conformal equivalence given  by the formula
$$
\varphi(z)=\frac{1}{2}+\frac{2i}{\pi}\log\frac{z+1}{1-z}
$$
which extends to the closed disc, except $z=\pm1$.\medskip

\begin{definition} A family $\mathcal X = \{X_\omega : \omega \in \partial \U\}$ of Banach spaces is an interpolation family with containing space $\Sigma$, intersection space $\Delta$ and containing function $k$ if:
\begin{itemize}
\item $\Sigma$ is a Banach space for which there are linear continuous embeddings $X_\omega \rightarrow \Sigma$. We will identify $X_\omega$ with its image in $\Sigma$ from now on.
\item $\Delta$ is a subspace of $\bigcap_{\omega \in \partial\U} X_\omega$ such that for every $x \in \Delta$ the function $z\in\T\bigcap\tilde\D \mapsto \|x\|_{\varphi(z)}$ is measurable and
\[
\int_\T \log^+ \|x\|_{\varphi(z)} d|z| < \infty,
\]
where $\log^+ t = \max(0, \log t)$ for $t>0$.
\item $k : \partial \U \rightarrow (0, \infty)$ is a measurable function such that
\[
\int_\T \log^+ k(\varphi(z)) d|z| < \infty,
\]
and $\|x\|_{\Sigma} \leq k(u)\|x\|_{u}$ for every $u \in \partial\U$ and every $x \in \Delta$.
\end{itemize}\end{definition}

If no risk of confusion arises we will simply say that $\mathcal X$ is an interpolation family. Given an interpolation family $\mathcal{X}$, we define $\mathscr{G} = \mathscr{G}({\mathcal{X}})$ as the space of all functions on $\U$ of the form $g = \sum_{j=1}^n g_j x_j$, where  $g_j \circ \varphi$ is in the Smirnov class $N^+$, $x_j \in \Delta$,  for all $j$, and
\begin{equation}\label{eq:esssup}
\|g\| = \esssup_{u \in \partial \U} \|g(u)\|_{u} < \infty.
\end{equation}
Here, $\partial\U$ carries the image of the measure $d|z|$ under the map $\varphi$.
Notice that this is well-defined because functions in the Smirnov class $N^+$ have a.\ e.\ nontangential limits on $\mathbb{T}$ and it does not depend of $ \varphi$, because if $\psi:\D\To\U$ is another conformal map, then $\varphi \circ \psi^{-1}$ is an automorphism of the disc.\medskip

Let us briefly explain how these spaces fit into the general framework described earlier. We just state the basic facts and refer the reader to \cite{Coifman1982, Correa2018, tesis} for more details. First, the evaluations $\delta_u:\G\To \Sigma$ are bounded. This fact depends on the hypotheses made on the containing function $k$. Indeed, by a result of Szeg\H{o} (see \cite[Proposition 1.1]{Coifman1982}), (any measurable extension of) the function $k\circ\varphi:\T\To\mathbb [0,\infty)$ has an associated ``outer'' function in the Smirnov class, which means that there is $K\in N^+$ such that $|K(z)|k(\varphi(z))=1$ for almost everywhere $z\in\T$, where the extension of $K$ to $\T$ is defined taking nontangential limits. It is now easy to check that for each $u\in \U$ one has $\|\delta_u:\G\To \Sigma\|\leq |K(z)|$, where $\varphi(z)=u$; see \cite[Proposition 2.3.52]{tesis}.\medskip

As a rule, the space $\G$ will fail to be complete; however it always fulfils conditions (a) and (b) in Definition~\ref{def:ad}:
Let $\F=\F(\mathcal X)$ be its completion and observe that the continuity of the evaluations at points of $\U$ allows us to identify $\F$ as a Banach space of analytic functions $\U\To \Sigma$, on which the point evaluations remain bounded with the same norm (see \cite[Proposition 2.3]{Coifman1982}). About condition (b), keeping an eye in Definitions~\ref{def:ad} and ~\ref{def:ac} let us observe the trivial fact that $\G$,
and therefore $\F$, are contractive modules over $H^\infty(\U)$, since every bounded analytic function on the disc belongs to the Smirnov class $N^+$. In particular, if $h:\U\To\D$ is a conformal map and $f\in\F$, then $hf\in\F$ and $\|hf\|_\F=\|f\|_\F$. It then remains to prove that $f\in\F$ whenever $hf\in\F$. This is related to the coincidence of the interpolation spaces associated to $\G$ and $\F$. To explain this, and following \cite{Coifman1982}, let us fix $z\in\U$ and consider the following two spaces: the first one, often denoted by $X{\{z\}}$, is the completion of the intersection space $\Delta$ equipped with the norm $
x\in\Delta\longmapsto\inf\{\|g\|_\F: g\in\G \text{ and } x=g(z)\}.$ The definition makes sense because for every $x\in\Delta$ there is $g\in\G$ such that $x=g(z)$. The other space is
$$
X{[z]}=\{x\in \Sigma: x=f(z) \text{ for some } f\in\F\},
$$
equipped with the quotient norm. To see how these spaces are related, have a look at the diagram
\begin{equation}
 \xymatrix{ 0 \ar[r] & \ker \delta_z \ar[r] & \F \ar[r] & \F/\ker \delta_z  \ar[r] & 0\\
 0 \ar[r] & \overline{\ker \delta_z \cap \G} \ar[u] \ar[r] & \overline{\G}\ar@{=}[u] \ar[r] & \F/\overline{\ker \delta_z \cap \G} \ar[u]^{Q}  \ar[r] & 0
 }
\end{equation}
Here, $Q\left( f+\overline{\ker \delta_z \cap \G}\right) = f+\ker \delta_z$ is an isometric quotient map (it maps the open unit ball of $\F/\overline{\ker \delta_z \cap \G} $ onto that of $\F/\ker \delta_z$). A moment's reflection suffices to realize that each nonzero element of $\ker Q$ corresponds to an element $f\in \ker \delta_z$ which is not in $\overline{\ker \delta_z \cap \G}$, and so $\ker Q=\ker \delta_z /\overline{\ker \delta_z\cap \G}$. Now observe that $X{\{z\}} = \F/\overline{\ker \delta_z \cap \G}$ and $X{[z]} = \F/\ker \delta_z$; thus, $Q$ induces an isometric quotient map of $X{\{z\}}\To X{[z]}$ ``extending'' the inclusion $\Delta\To \Sigma$. This map is injective (that is, $X{\{z\}}=X{[z]}$) if and only if $\ker \delta_z \cap \G$ is dense in $\ker \delta_z$. On the other hand, if $h:\U\To\D$ is a conformal map vanishing at $z$, we have $\ker \delta_z \cap\G=h\G$ in the sense that each function in $\G$ vanishing at $z$ has the form $h g$, for some $g\in \G$. Using this, the following lemma  is not hard to prove and it concludes the argument:

\begin{lemma}\label{lem:KGvsK} The following statements are equivalent:
\begin{itemize}
\item $X{\{z\}}=X{[z]}$.
\item $\ker \delta_z\cap \G$ is dense in $\ker \delta_z$.
\item $\ker \delta_z = h\cdot \F$.\hfill $\square$
\end{itemize}
\end{lemma}

\begin{definition} An interpolation family  $\mathcal X=\{X_\omega : \omega \in \partial \U\}$ is said to be admissible at $z\in \U$ if it satisfies the equivalent conditions recorded in the preceding Lemma, and it is said to be admissible if it is admissible at every $z\in\U$\end{definition}

Observe that an interpolation family $\mathcal X$ is admissible if and only if the space $\F$ obtained from it is admissible in the sense of Definition~\ref{def:ad}.

\section{Rochberg spaces and their entwining exact sequences}\label{sec:entwining}

Let us translate the basic facts of \cite{cck} to the context of acceptable spaces. If we fix $z\in\U$, the map $\delta_z:\F\to \Sigma$ is continuous and $\F/\ker{\delta_z}$ is a Banach space which is isometric to
$$
\F_z=\{w\in \Sigma: w=f(z) \text{ for some } f\in\F\},
$$
endowed with the quotient norm $\|w\|_{\F_z}=\inf_{w=f(z)}\|f\|_\F$. The family $(\F_z)_{z\in\U}$ will be called the analytic family of Banach spaces associated to $\F$, which is coherent with the traditional use when $\F$ is admissible (cf. \cite[\S~10]{kalt-mon}) and, in particular, when $\F$ arises from an admissible interpolation family, as in Section~\ref{sec:int-fam}. In this case we have $\F_z=X{[z]}=X\{x\}$.

The map $\delta_z^{(n)}:\mathscr F\To \Sigma$, evaluation of the $n$-{th} derivative at $z$, is bounded for all
$z\in \U$ and all $n\in \N$ by the boundedness of $\delta_z$, the
definition of derivative and the Banach-Steinhaus theorem.
Thus, it makes sense to consider the Banach spaces
\begin{equation}\label{eq:Fker}
\mathscr F/\bigcap_{i<n}\ker\delta_{z}^{(i)}\quad\quad(n\in\N).
\end{equation}
 As before these spaces are isometric to the Rochberg spaces
\begin{align}\label{eq:Xnz}
\nonumber \F^{(n)}_{z}&=\{w\in \Sigma^n: w=\tau_{(n,0]}f(z) \text{ for some } f\in\F\}\\
&=\left\{(w_{n-1},\dots,w_{0})\in \Sigma^n: w_i=	\frac{f^{(i)}(z)}{i!}\text{ for some } f\in \mathscr F \text{ and all } 0\leq i<n
\right\},
\end{align}
endowed with the quotient norm: the norm of $w=(w_{n-1},\dots,w_{0})$ in $
\F^{(n)}_{z}$ is the infimum of the norms of the functions of $\F$ fitting in (\ref{eq:Xnz}).\medskip

For fixed $z$, the spaces $\F^{(n)}_{z}$ can be arranged into exact sequences in a very natural way: this is implicit in \cite{rochpac}, even if the syntagma ``exact sequence" does not appear, and a complete treatment can be found in \cite{cck}. Indeed, if for $1\leq n,
k<m$ we denote by $\imath_{n,m}:\Sigma^n\to \Sigma^m$ the inclusion on the
left given by
$\imath_{n,m}(x_n,\dots,x_1)=(x_n,\dots,x_1,0\dots,0)$ and by
$\pi_{m,k}:\Sigma^m\to \Sigma^k$ the projection on the right given by
$\pi_{m,k}(x_m,\dots, x_k,\dots,x_1)=  (x_k,\dots,x_1)$, then
$\pi_{m,k}$ restricts to an isometric quotient map of
$
\F^{(m)}_{z}$ onto $
\F^{(k)}_{z}$ (this is trivial) and $\imath_{n,m}$ is an isomorphic embedding of
$
\F^{(n)}_{z}$ into $
\F^{(m)}_{z}$ (this can be proved as \cite[Proposition 2(a)]{cck}) and thus, see \cite[Theorem 4]{cck}, for each $n,k\in\mathbb N$ there is an exact sequence of Banach spaces and operators
\begin{equation}\label{eq:nn+kk}
\begin{CD}
0@>>> \F^{(n)}_{z} @>\imath_{n,n+k}>> \F^{(n+k)}_{z}@> \pi_{n+k,k} >> \F^{(k)}_{z}@>>> 0
\end{CD}
\end{equation}

To describe the sequences (\ref{eq:nn+kk}) as twisted sums we will use the maps $\Omega^{k,n}: \F^{(k)}_{z}\To \Sigma^n$ defined as follows: we fix $\e\in(0,1)$, and, for each $x=(x_{k-1},\dots,x_{0})$ in
$ \F^{(k)}_{z}$, select $f_x\in\F$ such that $x=\tau_{(k,0]}f_x(z)$, with $\|f_x\|\leq (1+\e)\|x\|$, in such a way that $f_x$ depends homogeneously on $x$. Then define
$$
\Omega^{k,n}(x)=\tau_{(n+k,k]}f_x(z).
$$
We could emphasize the fact that $\Omega^{k,n}$ depends on $z$ by adding the subscript $z$, if necessary. It is clear that it also depends on the choice of $f_x$, but different choices of $f_x$ only produce bounded perturbations of the same map.  Any $\Omega^{k,n}$
defined in this way is a quasilinear map (see the definition below) from $\F^{(k)}_{z}$ to $\F^{(n)}_{z}$, which means that there is a constant $C$ such that, for every $x,y\in \F^{(k)}_{z}$ the difference
$
\Omega^{k,n}(x+y)-
\Omega^{k,n}(x)-
\Omega^{k,n}(y)
$, which belongs a priori to $\Sigma^n$, actually falls into $\F^{(n)}_{z}$ and obeys the estimate
$$
\|\Omega^{k,n}(x+y)-
\Omega^{k,n}(x)-
\Omega^{k,n}(y)\|_{\F^{(n)}_{z}}\leq C\left(\|x\|_{\F^{(k)}_{z}}+ \|y\|_{\F^{(k)}_{z}}	\right).
$$
The map $\Omega^{k,n}$ can be used to form the twisted sum space
$$
{\F^{(n)}_{z}}\oplus_{\Omega^{k,n}}{\F^{(k)}_{z}}= \left\{(y,x)\in \Sigma^{n+k}: y-\Omega^{k,n}(x)\in {\F^{(n)}_{z}}, x\in {\F^{(k)}_{z}} \right\},
$$
endowed with the quasinorm
\begin{equation}\label{eq:eqnorm}
\|(y,x)\|_{\Omega^{k,n}}= \left\| y-\Omega^{k,n}(x)\right\|_{\F^{(n)}_{z}}+ \| x\|_{\F^{(k)}_{z}}.
\end{equation}
It turns out that ${\F^{(n)}_{z}}\oplus_{\Omega^{k,n}}{\F^{(k)}_{z}}$ and ${\F^{(n+k)}_{z}}$ are the same space, and that (\ref{eq:eqnorm}) is a quasinorm equivalent to the norm of  ${\F^{(n+k)}_{z}}$. Although the explicit use of quasilinear maps is marginal in this paper it will be convenient to record the definition here:

\begin{definition}\label{def:quasilinear}
Let $X$ and $Y$ be quasinormed spaces and let $H$ be a linear space containing $Y$. A homogeneous mapping $\Phi: X\To H$ (not $Y$) is said to be quasilinear from $X$ to $Y$ if:
\begin{itemize}
\item[(a)] $\Phi(x+y)-\Phi(x)-\Phi(y)\in Y$ for all $x,y\in X$.
\item[(b)] There is a constant $C$ such that  $\|\Phi(x+y)-\Phi(x)-\Phi(y)\|_Y\leq C\big(\|x\|_X+\|y\|_X\big)$ for all $x,y\in X$.
\end{itemize}
\end{definition}
Condition (a) guarantees that
$
Y\oplus_\Phi X=\{(h,x)\in H\times X: h-\Phi(x)\in Y\}
$
is a linear subspace of $H\times X$, while (b) and the homogeneous character of $\Phi$ imply that the formula
$
\|(h,x)\|_\Phi=\|h-\Phi(x)\|+\|x\|
$
defines a quasinorm on $Y\oplus_\Phi X$, which is equivalent to a norm when $\Phi$ arises as a derivation. The map $\imath:Y\To Y\oplus_\Phi X$ given by $\imath(y)=(y,0)$ preserves the (quasi) norms and the map $\pi:Y\oplus_\Phi X\To X$ given by $\pi(h,x)=x$ takes the unit ball of $Y\oplus_\Phi X$ onto that of $X$. These form a short exact sequence
\begin{equation}\label{eq:YZX}
\xymatrix{
0\ar[r] & Y \ar[r]^-\imath & Y\oplus_\Phi X \ar[r]^-\pi & X\ar[r] & 0
}
\end{equation}
  that shall be referred to as the sequence generated by $\Phi$. We say that $\Phi$ is trivial (as a quasilinear map from $X$ to $Y$), and we write $\Phi\sim 0$, if (\ref{eq:YZX}) splits, that is, if there is an operator $P: Y\oplus_\Phi X\To Y$ such that $P\,\imath={\bf I}_Y$, equivalently, there is an operator $S: X\To Y\oplus_\Phi X$ such that $\pi\,S={\bf I}_X$. This happens if and only if there is a, not necessarily continuous, linear map $L:X\To H$ such 
  that $\Phi-L$ is bounded from $X$ to $Y$ in the sense that $\|\Phi(x)-L(x)\|_Y\leq M\|x\|_X$ for some constant $M$ and all $x\in X$.
\medskip

 We conclude this rugged introduction emphasizing the compatibility of the sequences (\ref{eq:nn+kk}) passing through a given $\F^{(m)}_z$. Indeed, if  $m=k+n=i+j$, with $k<i$ say,
then the following diagram is commutative:
\begin{equation}\label{poz}
\begin{CD}
&&0&&0\\ &&@VVV @VVV\\
&&\F^{(j)}_z@= \F^{(j)}_z\\
&&@VVV @VVV \\
 0@>>> \F^{(n)}_z@>>>  \F^{(m)}_z@>>>  \F^{(k)}_z@>>>0\\
&&@VVV@VVV@|\\
 0@>>>\F^{(n-j)}_z@>>>  \F^{(i)}_z@>>>  \F^{(k)}_z@>>> 0\\
&&@VVV@VVV\\
&&0&&0
\end{CD}
\end{equation}

Notation has been lightened up by understanding that unlabelled arrows $\Fn_z\To \F^{(m)}_z$ are $\imath_{n,m}$ if $n\leq m$ and $\pi_{n,m}$ if $n\geq m$.\medskip

Sometimes it is convenient to replace the starting acceptable space $\F$ by another one with more convenient properties. If properly done, this will affect very little the resulting sequences:

\begin{lemma}\label{lem:Gzn=Fzn}
Let $\F$ and $\G$ be acceptable spaces of functions $\U\To \Sigma$. Assume that $\G\subset \F$ and that the inclusion is continuous. Fix $z\in\U$. If $\F_z=\G_z$, necessarily with equivalent norms, then $\Fn_z=\G_z^{(n)}$, with equivalent norms, for every $n\geq 1$ and the sequences induced by $\G$ at $z$ agree with those of $\F$.
\end{lemma}

The proof is an easy induction argument once one realizes that, under the hypothesis of the Lemma, given $n,k\geq 1$, there is a commutative diagram, of Banach spaces and operators
$$
\begin{CD}
0@>>> \G^{(n)}_{z} @>>> \G^{(n+k)}_{z}@>>> \G^{(k)}_{z}@>>> 0\\
&& @VVV @VVV @VVV\\
0@>>> \F^{(n)}_{z} @>>> \F^{(n+k)}_{z}@>>> \F^{(k)}_{z}@>>> 0
\end{CD}
$$
where the descending arrows are the corresponding formal inclusions. This implies that the derived spaces of admissible interpolation families and the corresponding exact sequences do not vary if one uses a norm of Hardy type instead of (\ref{eq:esssup}); see Pisier's comments in \cite{pisier}.

The most important application for us occurs in the context of couples, where one has $\CX_z^{(n)}=\CoX_z^{(n)}$, up to equivalent norms, for every $z$ in the corresponding strip and all $n\geq 1$.
Actually it is easy to see that $\CX_z^{(n)}$ and $\CoX_z^{(n)}$ are the same space, with the same norm, using functions of the form $w\mapsto\exp\big{(}\varepsilon(w-z)^{2+4k}\big{)}$, where $\varepsilon>0$ is small and $k\in\N$ is large. We will use this fact without further mention.

\section{The cornerstone example}\label{sec:corner}
 Let us investigate the particularly interesting case of the couple
$(\ell_\infty, \ell_1)$, which in a sense motivated the whole theory. We will denote by $\mathscr Z$ the Calder\'on space $\C(\ell_\infty, \ell_1)$ on the unit strip, that is, $\mathscr Z$ is the
space of analytic functions $f:\mathbb S\to\ell_\infty$ having the
following properties:
\begin{enumerate}
\item $f$ extends to a continuous function on $\overline{\mathbb S}\To \ell_{\infty}$ that we denote again by $f$.
\item $\|f\|_\mathscr Z=\sup\{\|f(it)\|_\infty,\|f(1+it)\|_1:t\in\mathbb R\}<\infty$.
\end{enumerate}
Of course $\mathscr Z$ is admissible and classical arguments show that $\mathscr Z_z$ is the complex interpolation space $[\ell_\infty,\ell_1]_{\theta}=\ell_p$ choosing $\theta=\Re z$ and $p=1/\theta$ for $\theta\in(0,1)$; in particular, $\mathscr Z_z=\ell_2$ for $z=1/2$. In the remainder of this Section we fix $z=1/2$ as the base point and we denote $\mathscr Z^{(n)}_{1/2}$ by $\mathscr Z_n$ for $n=1,2,\dots$ If $x$ is normalized in $\ell_2$ and we set $x=u|x|$ then $f_x(z)=u|x|^{2z}$ is normalized in $\mathscr Z$ and one has $f_x({1\over 2})=x$. Thus
$$
f_x(z)=u|x||x|^{2z-1}=x|x|^{2(z-1/2)}=x\sum_{n=0}^\infty\frac{2^n\log^n|x|}{n!}\left(z-\tfrac{1}{2}\right)^n,
$$
from where $(\tau_nf_x)( \tfrac{1}{2})= \tfrac{2^n}{n!} \,x\, \log^n|x|.$ Thus, for arbitrary $x\in\ell_2$ we have, by homogeneity,
\begin{equation*}
(\tau_nf_x)( \tfrac{1}{2}) =\frac{2^n x}{n!}\log^n\left(\frac{|x|}{\|x\|_2}\right).
\end{equation*}
Hence,
\begin{equation}\label{eq:omega1n}
\Omega^{1,n}(x)= \tau_{[n,1]}(f_x)( \tfrac{1}{2}) = x\left(\frac{2^{n}}{n!}\log^{n}\left(\frac{|x|}{\|x\|_2}\right),\dots,  \frac{2^2}{2!}\log^2\left(\frac{|x|}{\|x\|_2}\right), 2\log\left(\frac{|x|}{\|x\|_2}\right) \right),
\end{equation}
which leads to a quite manageable description of the spaces $\mathscr Z_m$. In particular, since $\mathscr Z_1=\ell_2$ we can use the map $\Omega^{1,1}$ to obtain that the functional
$
\|(y,x)\|_{\Omega^{1,1}}=\big\|y-2x\log\big(|x|/\|x\|_2\big)\big\|_2+\|x\|_2
$
is equivalent to the norm of  $\mathscr Z_2$. This shows that $\mathscr Z_2$ is isomorphic, but not equal, to the original Kalton-Peck space $Z_2$, whose quasinorm was defined by its legitimate owners as
$\|(y,x)\|_{Z_2}=\big\|y-x\log\big(\|x\|_2/|x|\big)\big\|_2+\|x\|_2.$
An isomorphism between both versions is $(y,x)\in Z_2\longmapsto (y, - 2x)\in\mathscr Z_2$.\medskip

The paper \cite{cck} contains a proof that $ \mathscr Z_m$ is not a subspace of a twisted Hilbert space for $m\geq 3$. We show now a general result which requires the following inductive, \emph{ad hoc} definition:

\begin{definition}
A twisted Hilbert space of order $1$ is just a Banach space which is isomorphic to a Hilbert space. For $k\geq 2$, say that $Z$ is a twisted Hilbert space of order $k$ if for some (equivalently, for every) choice $i+j=k$ with $i,j\geq 1 $ there is a short exact sequence $0\To A_i\To Z\To A_j\To 0$ in which $A_i$ (resp. $A_j$) is a twisted Hilbert space of order $i$ (resp. $j$).
\end{definition}

The equivalence between the ``for some" and the ``for every" form of the definition above is not entirely straightforward and requires a judicious use of diagrams: if $n+m=i$ the commutative diagram
$$
\begin{CD}
&&0&&0\\ &&@VVV @VVV\\
&&A_n@= A_n\\
&&@VVV @VVV \\
 0@>>> A_i@>>>  Z@>>>  A_j@>>>0\\
&&@VVV@VVV@|\\
 0@>>>A_m@>>>  Z/A_n@>>>  A_j@>>> 0\\
&&@VVV@VVV\\
&&0&&0
\end{CD}
$$
shows that a twisted Hilbert space $Z$ of order $i+j$ also decomposes as a twisted sum of twisted Hilbert spaces of order $n$ and $m+j$; and, analogously, if $n+m=j$ using the commutative diagram
$$
\begin{CD}
&&&&0&&0\\ &&&&@AAA @AAA\\
&&&& Z/B@= A_m@>>>0\\
&&&& @AAA @AAA\\
 0@>>>A_i@>>>  Z@>>>  A_j@>>> 0\\
&&@| @AAA@AAA\\
 0@>>>A_i@>>>  B@>>>  A_n@>>> 0\\
&&&&@AAA @AAA \\
&&&&0&&0
\end{CD}
$$
The space $\mathscr Z_m$ is a twisted Hilbert space of order $m$ and twisted sums of Hilbert spaces are exactly the twisted Hilbert spaces of order $2$.

\begin{theorem}\label{notsub} $\mathscr Z_m$ cannot be embedded into a twisted Hilbert space of order $k<m$. In particular $\mathscr Z_m$ is not a subspace of $ \mathscr Z_k$ whenever $k<m$.\end{theorem}

Let us recall from \cite{elp} the definition of $n$-th type 2 constant: If $X$ is a Banach space, $a_n(X)$ is the infimum of the constants $a$ such that
$$\underset{\pm}\Ave  \left \|\sum_{i=1}^n \pm x_i\right \|^2 \leq a^2 \sum_{i=1}^n \|x_i\|^2
$$
for all $x_1, \cdots, x_n \in X$. Theorem \ref{notsub} will follow straightforwardly from the next two Lemmata.
The first one generalizes estimates in \cite[Theorem~3]{elp}, \cite[Theorem~6.2]{kaltpeck}, \cite[Theorem 7.5]{kaltconv} that deal with the case $m=1$.

\begin{lemma} For each twisted Hilbert space $Z$ of order $m+1$ there is a constant $C$, depending on $Z$ and $m$, such that $a_n(Z) \leq C \log_2^{m} n$.
\end{lemma}
\begin{proof}
From \cite[Theorem 1, Part (1)]{elp} we know that, given a subspace $Y \subset Z$,  one has
$$a_{nk}(Z) \leq a_n(Y) a_k(Z) + a_n(Y) a_k(Z/Y) + a_n(Z) a_k(Z/Y)$$
We then proceed by induction. The result is trivial for $m=0$, by the parallelogram law. Assume it is true for twisted Hilbert spaces of order $m$.
Now let $Z$ be a twisted Hilbert space or order $m+1$ and let
$0\To H\To Z\To X\To 0$ be a witnessing sequence. There is no loss of generality if we assume that $Z$ contains $H$ isometrically and the corresponding quotient is $X$. Since $X$ is a twisted Hilbert space of order $m$ the induction hypothesis provides a constant $C$ such that $a_{n}( X) \leq C  \log_2^{m-1} n$ for all $n\in\N$.
Thus, for $k \in \mathbb{N}$, one has
$$
a_{2k}( Z) \leq a_2(H) a_k( Z) + a_2(H) a_k(X) + a_2( Z) a_k( X)=a_k( Z) + (1 + a_2( Z))a_k( X)\leq  a_k( Z) +  (1 + a_2( Z))C\log_2^{m-1} k
$$
and so
$$
a_{2^{n+1}}( Z) \leq a_{2^n}( Z) + (1 + a_2( Z))C \log_2^{m-1} 2^n = a_{2^n}( Z) + (1 + a_2( Z))C n^{m-1}.
$$
Also,
$$
a_{2^n}( Z)\leq a_{2^{n-1}}( Z)  + (1 + a_2( Z))C (n-1)^{m-1},
$$
and, iterating $n$ times, we obtain
$$
a_{2^{n+1}}( Z) \leq  a_2( Z) +  (1 + a_2(Z))C  \sum_{1\leq i\leq n} i^{m-1}.
$$
From Faulhaber's formula \cite[p. 108]{aig} we get that the dominating term of $\sum_{1\leq i\leq n} i^{m-1}$ is $n^{m}$. Using that $a_n$ is nondecreasing, there is some constant $C'$ such that $ a_{n}(Z ) \leq C' \log_2^{m} n$ for all $n$.\end{proof}

It is clear that if $Y$ is isomorphic to a subspace of $X$, then the sequence  $a_n(Y)/a_n(X)$ is bounded. The following computation completes the proof of Theorem~\ref{notsub}.

\begin{lemma}\label{lem:an}
For each $m\geq 0$ there is $c_m>0$ so that $c_m\log_2^{m} n \leq a_n( \mathscr Z_{m+1})$.
\end{lemma}

\begin{proof} Pick $s_N=\sum_{1\leq i\leq N}{ e_i}$. Since
$\|(0, \dots, 0, \sum_{i\leq N}{ e_i})\|= \|(0, \dots, 0, \sum_{i\leq N}{\pm e_i})\|$ and $\|\sum_{i\leq N}{ e_i}\|=\sqrt{N}$,
the inequality
$$\|(0, \dots, 0, s_N)\|_{\mathscr Z_{m+1}} \geq c_m \sqrt{N} \log_2^m N,$$
immediately yields the lower estimate for $a_n( \mathscr Z_{m+1})$.\medskip

We need the following elementary identity: for each $n\geq 1$ one has
\begin{equation}\label{eq:1/n!}
\frac{1}{n!}+ \frac{(-1)}{(n-1)!}+ \frac{(-1)^2}{2!(n-1)!}
+\dots+\frac{(-1)^{n-1}}{(n-1)!}+\frac{(-1)^n}{n!}=
\sum_{0\leq i\leq n }  \frac{(-1)^i}{i!} \frac{1}{(n-i)!}=0.
\end{equation}
This can be seen writing $1$ as the product $e^{-t}e^t$ and then using Leibniz rule to compute the $n$-th Taylor coefficient of the product at the origin.\medskip

For the rest of the proof we will use the following notations: given  $x\in \Sigma$ and scalars
$(\alpha_1,\dots,\alpha_k)$ we write $(\alpha_1,\dots,\alpha_k)x =(\alpha_1 x,\dots,\alpha_k x)$. Also, we set $L=\log(1/N)$. We also take advantage of the fact that each $\mathscr Z_{n+1}$ can be written as a twisted sum of  $\mathscr Z_n$ and $\ell_2$, using the map defined by (\ref{eq:omega1n}): taking $x=s_N$ there, we have
$$
\Omega^{1,n}(s_N)=
\left(\frac{2^{n}}{n!}\log^{n}\left(N^{-1/2}\right),\dots,  \frac{2^2}{2!}\log^2\left(N^{-1/2}\right), 2\log\left(N^{-1/2}\right) \right) s_N=
\left(\frac{L^{n}}{n!}, \frac{L^{n-1}}{(n-1)!},\dots,  \frac{L^2}{2!}, L \right) s_N.
$$
For each $n$ we fix a constant $k_n$ such that
$
\|(y,x)\|_{\mathscr Z_{n+1}}\geq k_n\left(\|(y-\Omega^{1,n}(x)\|_{\mathscr Z_{n}}+\|x\|_2\right)
$. 
Actually one can take $k_n={1\over 3}$ for all $n$. After this preparation:
\begin{align*}\|(0, \dots, 0, s_N)\|_{\mathscr Z_{m+1}}&\geq k_m\left\|\Omega^{1,m} (s_N)\right\|_{\mathscr Z_{m}}=
 k_m\left\| \left(\frac{L^{m}}{m!}, \frac{L^{m-1}}{(m-1)!},\dots,  \frac{L^2}{2!}, L \right) s_N \right\|_{\mathscr Z_{m}}\\
&\geq
k_m k_{m-1}\underbrace{\left\| \left(\frac{L^{m}}{m!}, \frac{L^{m-1}}{(m-1)!},\dots,  \frac{L^2}{2!}\right) s_N - \Omega^{1,m-1} (L s_N) \right\|_{\mathscr Z_{m-1}}}_{(\star)}.
\end{align*}
Now,
\begin{align*}
(\star)&=\Bigg\| \bigg(L^{m}\left[\tfrac{1}{m!}-  \tfrac{1}{(m-1)!}\right],
L^{m-1}\left[\tfrac{1}{(m-1)!}-  \tfrac{1}{(m-2)!}\right],\dots,
L^{3}\left[\tfrac{1}{3!}-  \tfrac{1}{2!}\right], L^{2}\overbrace{\left[\tfrac{1}{2!}-  \tfrac{1}{1!}\right]}^{-1/2!}
\bigg) s_N  \Bigg\|_{\mathscr Z_{m-1}}\\
&\geq k_{m-2}
\Bigg\| \bigg(L^{m}\left[\tfrac{1}{m!}-  \tfrac{1}{(m-1)!}+ \tfrac{1}{2!(m-2)!}\right],
L^{m-1}\left[\tfrac{1}{(m-1)!}-  \tfrac{1}{(m-2)!}+ \tfrac{1}{2!(m-3)!}\right],\dots,
L^{3}\underbrace{\left[\tfrac{1}{3!}-  \tfrac{1}{2!}+ \tfrac{1}{2!1!}\right]}_{1/3!}
\bigg) s_N  \Bigg\|_{\mathscr Z_{m-2}}.
\end{align*}
Continuing in this way, after $\ell$ iterations, we see that
$\|(0, \dots, 0, s_N)\|_{\mathscr Z_{m+1}}/(k_m\cdots k_{m-\ell})$ is at least
\begin{align*}
\Bigg\|\bigg{(}L^m &\left[\tfrac{1}{m!}-  \tfrac{1}{(m-1)!}+ \tfrac{1}{2!(m-2)!}+\dots+
\tfrac{(-1)^{\ell}}{\ell!(m-\ell)!}\right],
L^{m-1} \left[\tfrac{1}{(m-1)!}-  \tfrac{1}{(m-2)!}+ \tfrac{1}{2!(m-3)!}+\dots+
\tfrac{(-1)^{\ell}}{\ell!(m-1-\ell)!}\right],\\
&\quad\dots,
L^{\ell+2}\left[\tfrac{1}{(\ell+2)!}-  \tfrac{1}{(\ell+1)!}+ \tfrac{1}{2!\ell!}+\dots+
\tfrac{(-1)^{\ell}}{\ell!2!}\right],
L^{\ell+1}\underbrace{\left[\tfrac{1}{(\ell+1)!}-  \tfrac{1}{\ell!}+ \tfrac{1}{2!(\ell-1)!}+\dots+
\tfrac{(-1)^{\ell}}{\ell!1!}\right]}_{(-1)^{\ell}/(\ell+1)!}\bigg)s_N\Bigg\|_{\mathscr Z_{m-\ell}}.
\end{align*}
And letting $\ell=m-1$ we conclude that
$$
\frac{\|(0, \dots, 0, s_N)\|_{\mathscr Z_{m+1}}}{(k_m\cdots k_{1})}\geq
\Bigg\|\bigg{(}L^m \underbrace{\left[\tfrac{1}{m!}-  \tfrac{1}{(m-1)!}+ \tfrac{1}{2!(m-2)!}+\dots+
\tfrac{(-1)^{m-1}}{(m-1)!1!}\right]}_{(-1)^{m-1}/m!}s_N\bigg{)}\Bigg\|_{\mathscr Z_1}=\frac{|L|^m}{m!}N^{1/2}. \qedhere
$$
\end{proof}
A more general, less tortuous argument will be given in the final paragraph of Section~\ref{sec:duality}. Keep in mind in what follows that
$\mathscr Z_{n}\simeq \mathscr Z_{n}\oplus \mathscr Z_{n}$.
Indeed, if $\mathbb N=A\cup B$ is a partition of $\mathbb N$ into two infinite subsets, and we set $\mathscr Z_{n}(A)=\{(x_n,\dots,x_1)\in \mathscr Z_{n}: \operatorname{supp} x_i\subset A \text{ for all } 1\leq i\leq n\}$ and similarly for $\mathscr Z_{n}(B)$, then $\mathscr Z_{n}(A)$ and $\mathscr Z_{n}(B)$ are isometric to $\mathscr Z_{n}$ for all $n\geq 1$ exactly for the same reason as for $n=1$ and
 $\mathscr Z_{n}=\mathscr Z_{n}(A)\oplus \mathscr Z_{n}(B)$.

\begin{corollary} $ \mathscr Z_{n+k}$ is not isomorphic to  $\mathscr Z_n  \oplus  \mathscr Z_k$.
\end{corollary}
\begin{proof} Assuming $n\leq k,  \mathscr Z_n  \oplus  \mathscr Z_k$ is a subspace of $\mathscr Z_{k}\oplus \mathscr Z_{k} \simeq \mathscr Z_{k}$, while $ \mathscr Z_{n+k}$ is not.\end{proof}

\begin{corollary}\label{ab=nk} Let $0\leq k,j\leq n$. $ \mathscr Z_{n-k}\oplus \mathscr Z_{n+k} \simeq \mathscr Z_{n-j}  \oplus  \mathscr Z_{n+j}$ if and only if $k=j$.\end{corollary}
\begin{proof} Assume otherwise, and assume $j<k$. Then $\mathscr Z_{n+k}$ would be a subspace of $\mathscr Z_{n-j}  \oplus  \mathscr Z_{n+j}$, which is in turn a subspace of $\mathscr Z_{n+j}$, and that is impossible.
\end{proof}

It is likely that $ \mathscr Z_m$ does not contain complemented copies of $ \mathscr Z_n$ for $n<m$, which would imply that
$ \mathscr Z_{j}\oplus \mathscr Z_{k} \simeq \mathscr Z_{n}  \oplus  \mathscr Z_{m}$ if and only if $j=n, k=m$ or $j=m, k=n$.

\section{Duality issues}\label{sec:duality}

This Section studies the conjugate spaces of the Rochberg spaces associated to an admissible space and the corresponding (dual) exact sequences. The material presented here is closely related to \cite{Coifman1982, rochpac, Cw2014, cck} and has loose connections with \cite{cabecastdu, castmoredu, kaltpeck}. This section deals with spaces of analytic functions arising from admissible interpolation families. Let $\mathcal X$ be such a family, with spaces $(X_u)_{u\in\partial\U}$, containing space $\Sigma$, intersection space $\Delta$ and containing function $k:\partial\U\To(0,\infty)$. We also fix a conformal map $\varphi:\D\To\U$, as in Section \ref{sec:int-fam}. Let $\F=\F(\mathcal X)$ and $\G=\G(\mathcal X)$ be as in Section~\ref{sec:int-fam} and let us keep the traditional notation $X_z^{(n)}$ for $\Fn_z$, where $z\in\U$. When $n=1$ we just write $X_z$. It is an easy consequence from $\G$ being dense in $\F$ that $\Delta^n$ is dense in $X_z^{(n)}$ for all $z\in\U$ and all $n$. Besides, it follows from Lemma~\ref{lem:KGvsK} that for each $x\in\Delta^n$ and every $\e>0$ there is $g\in\G$ such that $x=\tau_{(n,0]}g(z)$ and $\|g\|_\F\leq (1+\e)\|x\|_{X_z^{(n)}}$. This simplification will play a role in the identification of the dual of $X_z^{(n)}$.

\subsection{Derivation of duals of interpolation spaces}\label{sec:W}
Adapting the techniques from \cite{Coifman1982} we may find the dual of the intermediate spaces $X_{z}$ the following way: let $\F^\bigstar$ be the space of functions $h : \U \rightarrow \Delta^\star$ (the algebraic dual of $\Delta$) such that
\begin{itemize}
\item $z \longmapsto \inp{h(\varphi(z))}{x}$ is a function in $N^+$ for every $x \in \Delta$;
\item there is $C > 0$ such that, for each $x\in\Delta$ one has $\lim_{w \rightarrow z} \left|\inp{h(\varphi(w))}{x}\right| \leq C\|x\|_{\varphi(z)}
$ for almost every $z \in \mathbb{T}$, where the limit is nontangential.
\end{itemize}
The space $\F^\bigstar$ will be normed taking $\|h\|_{\F^\bigstar}$ as the infimum of the numbers  $C$ satisfying the preceding condition. The question of whether  $\F$ is irrelevant for the subsequent discussion. For each $z\in\U$ there is an isometry between $X_{z}^*$ and the ``intermediate'' space
$$
({\F^\bigstar})_z = \{\xi\in\Delta^\star: \xi= h(z) \text{ for some } h \in \F^\bigstar\},
$$
with the natural quotient norm. More precisely, $\xi\in\Delta^\star$ belongs to $
({\F^\bigstar})_z$ if and only if the functional $x\in\Delta\longmapsto\langle\xi,x\rangle\in\mathbb C$ is bounded in the norm of $X_z$ in which case the norm of the obvious extension in $X_z^*$ agrees with the norm of $\xi$ in $({\F^\bigstar})_z$. We take this fact, proved in \cite[Theorem~3.1]{Coifman1982} when $\Delta$ is the whole intersection space, as the starting point of this section. From now on we identify $X_{z}^*$ with that subset of $\Delta^\star$, that is, we use $\Delta^\star$ as a ``containing space'' for the family $X_u^*$, with $u\in\U$. In this way the space $\F^\bigstar$ can be used to construct the derived spaces of the family $X_{z}^*$ using the ideas of Section~\ref{sec:entwining}.

First, we need a substitute for the derivatives: given $h\in \F^\bigstar$ and $n\geq 0$ we define $h^{(n)}: \U\To\Delta^\star$ by the formula
\[
\inp{h^{(n)}(z)}{x} =\frac{d^n}{dz^n}\:\inp{h(z)}{x}\quad\quad(x\in\Delta).
\]
The meaning of the expressions such as $\tau_A(h), \tau_A(h)(z)$, and the like should be obvious in this context. Now, set
$$
({\F^\bigstar})^{(n)}_z=\left\{(\xi_{n-1}, \dots, \xi_0)\in (\Delta^\star)^n: \text{ there is }  h\in \F^\bigstar\; \text{such that } \xi_i= \frac{h^{(i)}(z)}{i!} \text { for } 0\leq i<n	\right\},
$$
with the quotient norm. At this juncture most structural properties of the spaces $({\F^\bigstar})^{(n)}_z$ remain obscure: for instance if they are complete, or Hausdorff, or if $
({\F^\bigstar})^{(n)}_z$ contains $
({\F^\bigstar})^{(k)}_z$ when $k<n$. All these thrilling questions will we settled in the next section.

\subsection{Duality of the twisted sums}
The first part of the following result was proved by Rochberg for finite dimensional spaces in \cite[Theorem 4.1]{rochpac}.

\begin{proposition}\label{props-duality}
For each $z\in\U$ and each $n\geq 1$, there is a linear homeomorphism $T_n : ({\F^\bigstar})^{(n)}_z \longrightarrow (X^{(n)}_z)^*$ given by
\begin{equation}\label{eq:Tn}
T_n(\xi_{n-1}, \cdots, \xi_0)(x_{n-1}, \cdots, x_0) = \sum_{j=0}^{n-1} \langle \xi_j,x_{n-j-1} \rangle
\end{equation}
for $(\xi_{n-1}, \cdots, \xi_0)\in ({\F^\bigstar})^{(n)}_z$ and $x_j\in\Delta, 0\leq j <n$. In particular, $({\F^\bigstar})^{(n)}_z$ is a Banach space. Moreover,
\begin{equation}\label{eq:||Tn||}
\left\| T_n : ({\F^\bigstar})^{(n)}_z \longrightarrow (X_z^{(n)})^*\right\|\leq \frac{1}{\dist(z,\partial\U)^{n-1}}.
\end{equation}
\end{proposition}

As the reader may guess, the lion's share of the proof is the boundedness of the pairing (\ref{eq:Tn}).
We shall need a number of intermediate steps, some new notations and a bit of function theory.


Given integers $n$ and $k$, we consider the maps
$\jmath_{n,n+k}: (\Delta^\star)^n\To (\Delta^\star)^{n+k}$ and
$\varpi_{n+k,k}: (\Delta^\star)^{n+k}\To (\Delta^\star)^{k}$  defined by
$$
\jmath_{n,n+k}((\xi_{n-1}, \dots, \xi_0))= (\xi_{n-1}, \dots, \xi_0, \underbrace{0, \ldots ,
            0}_{k \text{ times}})\quad\text{and}\quad
            \varpi_{n+k,k}(\xi_{n+k-1}, \dots, \xi_k, \xi_{k-1}, \dots, \xi_0)= (\xi_{k-1}, \dots, \xi_0)
$$
We label them this way to distinguish them from the maps $\imath_{n,n+k}: \Sigma^{n}\To \Sigma^{n+k}$ and $\pi_{n+k,k} :\Sigma^{n+k}\To \Sigma^{k}$ appearing in (\ref{eq:nn+kk}), although they are formally the same maps.

\begin{lemma}
For every $n, k \geq 1$ and every $z\in \U$, the map $\jmath_{n, n+k}$ is bounded from $({\F^\bigstar})^{(n)}_z$ to $({\F^\bigstar})^{(n+k)}_z$ and  $\varpi_{n+k,k}$ is an isometric quotient map from $({\F^\bigstar})^{(n+k)}_z$ to $({\F^\bigstar})^{(k)}_z$. 
\end{lemma}
\begin{proof}
By \cite[Lemma~1]{cck} there is a polynomial $P$ of degree at most $n+k-1$ such that
$(P\circ\varphi^{-1})^{(i)}(z)=i!\delta_{ik}$ (Kronecker delta) for $0\leq i<n+k$.
Pick $\xi=(\xi_{n-1}, \cdots, \xi_0)$ in $({\F^\bigstar})^{(n)}$ and $h\in {\F^\bigstar}$ such that $\tau_{(n,0]}h(z)=\xi$.
Consider the function $H=(P\circ\varphi^{-1})\dot h$. Then $H\in {\F^\bigstar}$,
$$
\tau_{(n+k-1,0]}H(z)=(\xi_{n-1}, \cdots, \xi_0, 0,\ldots, 0)= \jmath_{n,n+k}(\xi),
$$
and $\|H\|_{\F^\bigstar}\leq \left(\sum_i|a_i|\right)\|h\|_{\F^\bigstar}$, where $a_i$ are the coefficients of $P$, so that $\|\jmath_{n,n+ k}:({\F^\bigstar})^{(n)}_z\To ({\F^\bigstar})^{(n+k)}_z\| \leq \sum \left|a_i\right|$. The second part is trivial.\end{proof}

\begin{lemma}\label{lem-h}
Let $h \in {\F^\bigstar}$ and $g \in \mathscr{G}$. Then the function $f : \U \To \mathbb{C}$ given by $f(z) = \langle h(z), g(z)\rangle$
is bounded, analytic on $\U$ and one has
\begin{equation}\label{eq-derivative-h}
f^{(n)}(z)=
 n! \sum\limits_{j=0}^n \left\langle\frac{h^{(j)}(z)}{j!} , \frac{g^{(n-j)}(z)}{(n-j)!}\right\rangle\quad\quad\text{and}\quad\quad\frac{\left|f^{(n)}(z)\right|}{n!}\leq \frac{\|h\|_{\F^\bigstar} \|g\|_\F}{\dist(z,\partial \U)^n}.
\end{equation}
\end{lemma}

\begin{proof}
We begin by noticing that by our assumptions, and the very definition of $\G$, the composition $f \circ \varphi$ is in $N^+$, and therefore has  almost everywhere nontangential limits on $\mathbb{T}$.
If we denote by $F$ the boundary values of $f \circ \varphi$, we  have $|F(z)| \leq \|h\|_{\F^\bigstar} \|g\|_\F$ for almost every  $z\in\mathbb T$, so that $z\in\T \mapsto F(z)$ is in $L_{\infty}(\T)$. This implies that $f\circ\varphi \in H^{\infty}$, and therefore $f$ is bounded on $\U$.

We will establish (\ref{eq-derivative-h}) by induction on $n\geq 0$.
The initial step ($n=0$) is the definition of $f$. Suppose (\ref{eq-derivative-h}) is valid for a given $n\geq 0$, rewrite it as $ f^{(n)}(z) = \sum_{j=0}^n \binom{n}{j} \langle h^{(j)}(z), g^{(n-j)}(z) \rangle$, and let us check the induction step:
\begin{align*}
\frac{d}{dz} f^{(n)}(z) & =
\sum\limits_{j=0}^n \binom{n}{j} \frac{d}{dz}\langle h^{(j)}(z), g^{(n-j)}(z)\rangle \\
&= \sum\limits_{j=0}^n \binom{n}{j} \left( \langle h^{(j+1)}(z), g^{(n-j)}(z) \rangle + \langle h^{(j+1)}(z), g^{(n-j + 1)}(z) \rangle\right) \\
	& =  \sum\limits_{j=1}^{n+1} \binom{n}{j-1} \langle h^{(j)}(z), g^{(n + 1 -j)}(z) \rangle + \sum\limits_{j=0}^n \binom{n}{j} \langle h^{(j)}(z), g^{(n + 1 -j)}(z) \rangle \\
    & =  \binom{n}{0} \langle h^{(0)}(z), g^{(n + 1)}(z) \rangle + \sum\limits_{j=1}^n \binom{n+1}{j} \langle h^{(j)}(z), g^{(n + 1 -j)}(z) \rangle
     + \binom{n}{n} \langle h^{(n+1)}(z), g^{(0)}(z) \rangle \\
    & =  \sum\limits_{j=0}^{n+1} \binom{n+1}{j} \langle h^{(j)}(z), g^{(n + 1 -j)}(z) \rangle
\end{align*}
The estimate follows from the bound $|f(u)|\leq \|h\|_{\F^\bigstar} \|g\|_\F$ for all $u\in\U$ and Cauchy's estimates, taking into account that for every $r<\dist(z,\partial\U)$ the disc of radius $r$ centered at $z$ lies inside $\U$.
\end{proof}

\begin{proof}[Proof of Proposition \ref{props-duality}]
We begin by showing that, for each $z\in\U$ the map $T_n$ is bounded from $({\F^\bigstar})^{(n)}_z$ to $(X^{(n)}_z)^*$. Put $x=(x_{n-1}, \cdots, x_0)$ and $\xi=(\xi_{n-1}, \cdots, \xi_0)$. Take $g \in \mathscr{G}$ such that $\tau_{(n,0]}g(z)=x$ and
 a corresponding $h \in \mathscr \Sigma$ for $\xi$. Let $f(u) = \langle h(u), g(u)\rangle$. By Lemma \ref{lem-h}, $f$ is bounded and analytic on $\U$, with
$|f(u)| \leq \|h\|_{\F^\bigstar}\|g\|_\G$ for all $u \in \U$ and
$$
\left|T_{(n)}(\xi)(x)\right|=\frac{\left|f^{(n-1)}(z)\right|}{(n-1)!}\leq \frac{ \|h\|_{\F^\bigstar}\|g\|_\G}{\dist(z,\partial\U)^{n-1}}.
$$
Since $h$ and $g$ were arbitrary, we obtain that $T_{(n)}(\xi)$ extends to a continuous functional on $X^{(n)}_z$ that we call again $T_{(n)}(\xi)$, and that $T_{(n)}$ is a bounded map, with $\| T_n : ({\F^\bigstar})^{(n)}_z \longrightarrow (X^{(n)}_z)^*\|\leq \dist(z,\partial\U)^{1-n}$.\medskip

The remainder of the proof is easier. First, for $n,k\geq 1$ and $z\in\U$, the following diagram is commutative:
 \begin{equation}\label{eq-diagram-Tn}
\xymatrixcolsep{3.5pc}  \xymatrix{ 0 \ar[r] & ({\F^\bigstar})^{(n)}_z \ar[r]^{\jmath_{n, n+k}}\ar[d]^{T_n} &  ({\F^\bigstar})^{(n+k)}_z
    \ar[r]^{\varpi_{n+k, k}}\ar[d]^{T_{n+k}} & ({\F^\bigstar})^{(k)}_z\ar[r]\ar[d]^{T_{k}} &0 \\ 0 \ar[r] & (X_z^{(n)})^*
    \ar[r]^{{\pi_{n+k, k}}^*} & (X_z^{(n+k)})^* \ar[r]^{{\imath_{k, n+k}}^*} & (X_z^{(k)})^*\ar[r] &0}
  \end{equation}

At this stage of the proof we cannot guarantee the exactness of the upper row of the preceding diagram: we have not proved that the image of $\jmath_{n, n+k}$ fills the kernel of $\varpi_{n+k, k}$.
 However, we know that $T_1$ is an isomorphism (it is in fact an isometry, by the result of Coifman, Cwikel, Rochberg, Sagher and Weiss mentioned before) and then a diagram chasing argument quickly shows that $T_m$ is an isomorphism for all $m\geq 1$. Indeed let us assume that $T_n$ and $T_k$ are isomorphisms and let us check that then so is $T_{n+k}$. It is clear that $T_{n+k}$ is injective. We show that it is also onto and open. Pick an arbitrary $x^*\in (X_z^{(n+k)})^*$ and let $\xi\in ({\F^\bigstar})^{(n+k)}_z$ be such that ${\varpi_{n+k, k}}(\xi)= T_{k}^{-1}({\imath_{n, n+k}}^*(x^*))$, with
$$
\|\xi\|_{({\F^\bigstar})^{(n+k)}_z} \leq C \| T_{k}^{-1}({\imath_{n, n+k}}^*(x^*)) \|_{({\F^\bigstar})^{(k)}_z}.
$$
for a constant $C$ independent of the choices. Now, $x^*-T_{n+k}(\xi)$ belongs to $\ker {{\imath_{n, n+k}}^*}$ and since the lower row is exact there is $y^*\in X^{(n)}_z$ such that ${{\pi_{n+k, k}}^*}(y^*)=x^*-T_{n+k}(\xi)$.
Letting $\eta= \jmath_{n, n+k}(T_n^{-1}(y^*))$ it is clear that $x^*=T_{n+k}(\xi+\eta)$.
Besides,
\begin{align*}
\|\xi\|&\leq C\|T_k^{-1}\|\:\|\imath_{k,n+k}\|\:\|x^*\|,\\
\|\eta\|&\leq \|\jmath_{n, n+k}\|\:\|T_n^{-1}\|\: \| ({{\pi_{n+k, k}}^*})^{-1}\|\left(1+\|T_{n+k}\|\right)\|x^*\|.\qedhere
\end{align*}
\end{proof}
Once thus has:

\begin{theorem}\label{thm-Tn-is-isomorphism}
For every $n,k\geq 1$ and each $z\in\U$ there is a commutative diagram
  $$\xymatrixcolsep{3.5pc}\xymatrix{ 0 \ar[r] & ({\F^\bigstar})^{(n)}_z \ar[r]^{\jmath_{n, n+k}}\ar[d]^{T_n} &  ({\F^\bigstar})^{(n+k)}_z
    \ar[r]^{\varpi_{n+k, k}}\ar[d]^{T_{n+k}} & ({\F^\bigstar})^{(k)}_z\ar[r]\ar[d]^{T_{k}} &0 \\ 0 \ar[r] & (X_z^{(n)})^*
    \ar[r]^{{\pi_{n+k, k}}^*} & (X_z^{(n+k)})^* \ar[r]^{{\imath_{k, n+k}}^*} & (X_z^{(k)})^*\ar[r] &0}
    $$
in which the vertical arrows are linear homeomorphisms and the rows are exact.
\end{theorem}

\subsection{A useful ``norming'' subspace to work with couples} In this Section we take advantage of a result by Cwikel \cite[Theorem~3.1]{Cw2014} to obtain a quite useful subspace of the dual space of the derived spaces of a couple.

Let $(X_0,X_1)$ be a Banach couple with sum $\Sigma$ and intersection $\Delta$, which is equipped with the norm $x\in X_0\cap X_1\longmapsto \max\big{(} \|x\|_0, \|x\|_1\big{)}$. We assume that $(X_0,X_1)$ is  regular according to Cwikel \cite{Cw2014}, i.e., $\Delta$ is dense in each $X_i$. Then each $X_i^*$ embeds into $\Delta^*$ (not $\Delta^\star$) in such a way that $X_0^*\cap X_1^*=\Sigma^*$.

There is a natural bilinear pairing $B:\CoX\times \C(X_0^*, X_1^*)\To A(\St)$ defined by
$$
B(h,g)(z) =\langle h(z), g(z)\rangle$$
for $g\in\CooX, h\in \C(X_0^*, X_1^*)$; recall that such a $g$ takes values in $\Delta$ and that $\CooX$ is dense in $\CoX$ so that the previous $B$ can be extended to the completion), where the brackets refer to the duality between $\Delta^*$ and $\Delta$. Now, \emph{mutatis mutandis}, the arguments of the preceding section yield:

\begin{proposition}
\label{prop:norming}
For each $z\in\St$ and each $n\geq 1$, let $T_n:
\C(X_0^*, X_1^*)_z^{(n)}\To \big{(}\CoX_z^{(n)}\big{)}^{*}$ be given by
\begin{equation}\label{eq:Tncouples}
T_n(\xi_{n-1}, \cdots, \xi_0)(x_{n-1}, \cdots, x_0) = \sum_{j=0}^{n-1} \langle \xi_j,x_{n-j-1}\rangle
\end{equation}
for $(\xi_{n-1}, \cdots, \xi_0)\in \C(X_0^*, X_1^*)_z^{(n)} , x_j\in\Delta, 0\leq j <n)$. The operator $T_n$ is bounded, with
\begin{equation}\label{eq:||Tn||couples}
\left\| T_n : \C(X_0^*, X_1^*)_z^{(n)}  \longrightarrow \big{(}\CoX_z^{(n)}\big{)}^{*}\right\|\leq \frac{1}{\dist(z,\partial\St)^{n-1}}.
\end{equation}
Moreover, $T_n$ ``renorms'' $\CoX_z^{(n)}$ in the following sense: there exist constants $c, C>0$ that depend on $z$ and $n$ such that
\begin{equation}\label{eq:renorming}
c\, \|x\|_{\CoX_z^{(n)}}\leq \sup\left\{ \big{|}T_n(\xi)(x)\big{|}: \xi\in \C(X_0^*, X_1^*)_z^{(n)}, \|\xi\|_{\C(X_0^*, X_1^*)_z^{(n)}}\leq 1\right\} \leq C\, \|x\|_{\CoX_z^{(n)}}.
\end{equation}
The constants do not depend on $\xi$ or $x$. In particular, if $X_\zeta$ is reflexive for some $0<\Re(\zeta)<1$, which is always the case if one of the spaces of the couple is reflexive, then $T_n$ is an isomorphism for every $n$ and $z$. The same happens if $X_0$ or $X_1$ is an Asplund space (equivalently, the dual has the Radon-Nikod\'ym property).
\end{proposition}

\begin{proof}[Sketch of the Proof]
The proof of the first part runs parallel to that of Proposition~\ref{props-duality} and is left to the reader. The ``moreover'' part
follows from Cwikel's result mentioned earlier (namely, that when $n=1$ the inequalities
in (\ref{eq:renorming}) are actually equalities with $c=C=1$) by an easy induction argument. Consider the commutative diagram
\begin{equation*}
\xymatrixcolsep{3pc}  \xymatrix{ 0 \ar[r] & \C(X_0^*, X_1^*)_z^{(n)} \ar[r] \ar[d]^{T_n} &  \C(X_0^*, X_1^*)_z^{(n+k)}
    \ar[r] \ar[d]^{T_{n+k}} & \C(X_0^*, X_1^*)_z^{(k)}\ar[r]\ar[d]^{T_{k}} &0 \\ 0 \ar[r] & (\CX_z^{(n)})^*
    \ar[r]^{{\pi_{n+k, n}}^*} & (\CX_z^{(n+k)})^* \ar[r]^{{\imath_{k, n+k}}^*} & (\CX_z^{(k)})^*\ar[r] &0}
  \end{equation*}
and recall our convention about unlabelled arrows. Assuming that $T_n$ and $T_k$ are ``renormings'', one quickly obtains chasing the diagram $T_{n+k}$ renorms $\CX_z^{(n+k)}$. The last assertion in the statement follows from being $T_1$ a surjective isometry \cite[Theorem 4.4]{kalt-mon}, as it was explained during the proof of Proposition \ref{props-duality}, and thus, by Diagram \ref{eq-diagram-Tn}, the same occurs to all $T_n$.\end{proof}

\begin{corollary}
For every $z\in\St$ and every $n\geq 1$ the dual of $\C(\ell_\infty,\ell_1)_z^{(n)}$ is isomorphic to $\C(\ell_\infty,\ell_1)_{1-z}^{(n)}$.
\end{corollary}

If we specialize to $z={1\over 2}$ we obtain that each of the spaces $\mathscr Z_n$ is isomorphic to its dual (Thus, for instance, Theorem~\ref{notsub} can be dualized replacing ``embeds in'' by ``is a quotient of'', and so on). However the pairing witnessing it is not
$$
\langle (y_{n-1},\dots, y_0), (x_{n-1},\dots, x_0) \rangle =\sum_{i+j=n-1} \langle y_i, x_j \rangle
$$
because this pairing induces an isomorphism between $\C(\ell_\infty,\ell_1)_{1/2}^{(n)}=\mathscr Z_n$ and the dual of $\C(\ell_1, \ell_\infty)_{1/2}^{(n)}$ which is isometric, but not equal, to $\mathscr Z_n$. In general, an isometry between
$\C(X_0,X_1)_{1/2}^{(n)}$ and $\C(X_1,X_0)_{1/2}^{(n)}$ can be obtained as follows: pick $x=(x_{n-1},\dots,x_0)$ in $\C(X_0,X_1)_{1/2}^{(n)}$ and then $f\in \C(X_0,X_1)$ such that $x=\tau_{(n,0]}f({1\over 2})$ and $\|f\|\leq \|x\|+\varepsilon$. Clearly $g(z)=f(1-z)$ has the same norm in $\C(X_1,X_0)$ as $f$ in $\C(X_0,X_1)$. Hence  $\tau_{(n,0]}g({1\over 2})$ belongs to  $\C(X_1,X_0)_{1/2}^{(n)}$ and has the same norm as $x$. Clearly
$$
\tau_{(n,0]}g(\tfrac{1}{2})=\big( (-1)^{n-1}x_{n-1},\dots,-x_1, x_0	\big).
$$
The inexorable conclusion is that the pairing that defines the isomorphism between $\mathscr Z_n$ and its dual is
$$
\langle (y_{n-1},\dots, y_0), (x_{n-1},\dots, x_0) \rangle =\sum_{i+j=n-1} (-1)^i\langle y_i, x_j \rangle
$$
If we denote by $u_n:\mathscr Z_n\To \mathscr Z_n^*$ the corresponding (noncanonical) isomorphism then the family $(u_n)_{n\geq 1}$ is ``almost'' compatible with the natural exact sequences:


\begin{corollary}
With the same notations as before, for every $k,n\geq 1$ the following diagram is commutative
$$\begin{CD} 0@>>> \mathscr Z_n @>\imath_{n,n+k}>> \mathscr Z_{n+k} @>\pi_{n+k, k}>> \mathscr Z_k @>>> 0\\
&&@V{(-1)^ku_n}VV @VVu_{n+k}V @VV{u_k}V\\
0@>>> \mathscr Z_n^* @> \pi_{n+k, k}^*>> \mathscr Z_{k+n}^* @>\imath_{n,n+k}^*>> \mathscr Z_k^* @>>> 0\end{CD}$$
\end{corollary}

The continuity of the operators $T_n$ of Proposition~\ref{prop:norming} provides lower bounds for the norm of an element of the form $(0,\dots,0,x)$ in $X^{(n)}_c=\CX^{(n)}_c$. Note that for $0<c<1$ we have $\dist(c,\partial\St)=\min(c,1-c)$.
Now, if $h\in\Cdual$, and $x\in\Delta$, then
\begin{equation}\label{eq:lowerbound}
\big{|}\langle \tau_{n-1}h(c), x \rangle\big{|} =  \big{|}\big(T_n  \tau_{(n,0]}h(c)\big)(0,\dots,0,x)\big{|} \leq \frac{\|h\|_{\Cdual}\, \|(0,\dots,0,x)\|_{X^{(n)}_c}}{\min(c,1-c)^{n-1}}.
\end{equation}
Let us consider again the case where $X_0=\ell_\infty$ and $X_1=\ell_1$ and estimate the norm of $(0,\dots,0,s_N)$ in the space $\mathscr Z^{(n)}_c=\mathscr C(\ell_\infty,\ell_1)^{(n)}_c$, for $0<c<1$. Note that $\mathscr Z^{(1)}_c=\ell_p$ with $p=1/c$ and so $\|s_N\|_{\mathscr Z^{(1)}_c}=N^{1/p}=N^c$.
If we interpret $\mathscr C(c_0,\ell_1)$ as a subset of  $\mathscr C(\ell_\infty,\ell_1)$ in the obvious way, $\mathscr C(\ell_\infty,\ell_1)_c= \mathscr C(c_0,\ell_1)_c$, with the same norm: just think of the finitely supported sequences. It follows from Lemma~\ref{lem:Gzn=Fzn} that $\mathscr C(\ell_\infty,\ell_1)^{(n)}_c= \mathscr C(c_0,\ell_1)^{(n)}_c$ for all $n\geq 1$ and $0<c<1$, still with the same norm.
Since $(c_0,\ell_1)$ is a regular couple we can go to Proposition~\ref{prop:norming} and then compute the extremals in $\mathscr C(\ell_1, \ell_\infty)$. Note that
$\mathscr C(\ell_1, \ell_\infty)_c=\ell_q$, where $q$ is the conjugate exponent of $p$ and that if $x$ is positive and normalized in $\ell_q$, then the function $z\mapsto x^{q(1-z)}$ is normalized in $\mathscr C(\ell_1, \ell_\infty)$ and assumes the value $x$ at $z=c$. It follows that for any $x\in\ell_q$  the function
$$
h(z)= x\left(\frac{|x|}{\|x\|_q}	\right)^{-q(z-c)}
= x\sum_{n\geq 0}\frac{(-q)^n}{n!}\log^n\left(\frac{|x|}{\|x\|_q} \right) (z-c)^n
$$
is an extremal for $x$ in $\mathscr C(\ell_1, \ell_\infty)$, with
$$
\tau_n h(c)=  \frac{(-q)^n}{n!}\, x\, \log^n\left(\frac{|x|}{\|x\|_q} \right).
$$
Letting $x=s_N$ in $\ell_p$ and taking $h$ as the corresponding extremal for $s_N$ in $\mathscr C(\ell_1, \ell_\infty)$ so that $h(c)=s_N$, with $\|h\|_{\mathscr C(\ell_1, \ell_\infty)}=N^{1/q}$ and applying (\ref{eq:lowerbound}) one obtains
$$
\frac{N\log^{n-1} N}{(n-1)!} \leq
\frac{ N^{1/q} \|(0,\dots,0,s_N)\|_{X^{(n)}_c}}{\min(c,1-c)^{n-1}},
$$
hence (compare with the proof of Lemma~\ref{lem:an})
$$
\|(0,\dots,0,s_N)\|_{X^{(n)}_c}\geq   \frac{\min(c,1-c)^{n-1}}{(n-1)!}
{N^{1/p}\log^{n-1} N} .
$$

\section{Analytic families of Rochberg spaces and interpolation}\label{secondaxis}\label{sec:disc}

This section develops the central topic of the paper and it is where acceptable spaces are required and admissible spaces do not suffice. The domain $\U$ on which an acceptable space of analytic functions is based plays an important role here. The simplest domains are: the unit strip $\St$, where classical interpolation for couples occurs, and the unit disk $\D$, where classical interpolation for families occur. Thus, to motivate the problem let us consider first:

\subsection{The case of couples}\label{sec:couples} The following reiteration-like result is so natural that we can hardly believe it has not been explicitly stated elsewhere.

\begin{proposition}\label{prop:couples}
Let $(X_0,X_1)$ be a regular compatible couple of Banach spaces on the strip $\mathbb S$, with sum $\Sigma$, intersection $\Delta$ and $0<a<b<1$. For every $n\geq 1$ the Rochberg spaces $X^{(n)}_a$ and $X^{(n)}_b$ form a compatible couple on the strip $\St_{a,b}$ as subspaces of $\Sigma^n$ and, for every $a<c <b$, the formal inclusion $X^{(n)}_c\To [X^{(n)}_a, X^{(n)}_b]_c$ is an isomorphic embedding. If, in addition, $\Delta^n$ is dense in $X^{(n)}_a\cap X^{(n)}_b$, which is always the case when $X_1$ contains $X_0$, then $[X^{(n)}_a, X^{(n)}_b]_c= X^{(n)}_c$, with equivalent norms.
\end{proposition}
\begin{proof} We first remark that in the case of couples we may assume that the norm of $\Sigma$ is majorized by those of $X_0$ and $X_1$. Thus, integrating on  large rectangular contours and using Cauchy integral formul\ae\, one gets, for $0<\theta <1$, that
$$\|\delta^{(n)}_\theta:\CoX\To \Sigma\|\leq \frac{n!}{\min(|\theta|,|1-\theta|)^n}.
$$
Thus, if $x=(x_{n-1},\dots,x_0)$ belongs to $X^{(n)}_{\theta}$, and $f\in\mathscr C(X_0,X_1)$ is such that $x=\tau_{(n,0]}f(\theta)$,
then
$$
\max_{0\leq i<n}\|x_i\|_\Sigma\leq \frac{\|f\|_\mathscr C}{\min(|\theta|, |1-\theta|)^n},
$$
hence $\Sigma^n$ contains both $X^{(n)}_{a}$ and $X^{(n)}_{b}$, the inclusions are continuous and $\big{(}X^{(n)}_{a}, X^{(n)}_{b}\big{)}$ is a compatible couple ready for interpolation on the strip $\St_{a,b}$. From now on, we write $Y_a=X^{(n)}_{a}$  and $Y_b=X^{(n)}_{b}$. Notice that at the moment we do not know whether $Y_c=X^{(n)}_{c}$, which is the conclusion of the Theorem. We end this preparation noticing that, according to our general notations,
$$
X^{(n)}_\eta=\CX_\eta^{(n)}=\CoX_\eta^{(n)}\quad\quad \text{and} \quad\quad Y_c= [Y_a, Y_b]_c=\mathscr C(Y_a, Y_b)_c= \mathscr C_0(Y_a, Y_b)_c.
$$
Let us see that $X^{(n)}_c\subset Y_c$ with contractive inclusion, which is the easy part. Given $f\in\CX$ we define an analytic function $R(f): \St_{a,b} \rightarrow \Sigma^n$ by $Rf(z)=\tau_{(n,0]}f(z)$. We claim that $R$ defines a bounded operator from $\CoX$ to $\mathscr C(Y_a, Y_b)$. Clearly, if $f$ is a simple function with values in $\Delta$ then $Rf\in \mathscr C(Y_a, Y_b)$ and $\|Rf\|_{\mathscr C(Y_a, Y_b)}\leq \|f\|_{\CX}$. For arbitrary $f\in\CoX$ the claim follows from an obvious density argument. We therefore have a commutative square
$$\xymatrix{ \CoX \ar[d]_{\delta_c\circ\tau_{(n,0]}} \ar[r]^R & \mathscr C(Y_a, Y_b) \ar[d]^{\delta_c}\\
 X_c^{(n)} \ar[r]^{\text{identity}} & Y_c
 }
$$
witnessing that the formal identity is a bounded operator from  $X_c^{(n)}$ to $Y_c$ with norm at most 1. To complete the proof of the first part we must show that there is a constant $C$ such that $\|x\|_{X_c^{(n)}}\leq C\|x\|_{Y_c}$ for $x\in \Delta^n$. We need here the duality results of the preceding Section. Since $T_n$ renorms $X^{(n)}_c$, it suffices to show that there is a constant $K$ such that
$$
|T_n\xi(x)|\leq K\|\xi\|_{\mathscr C(X_0^*,X_1^*)^{(n)}_c} \|x\|_{Y_c}$$
for $x\in\Delta^n, \xi\in \mathscr {\mathscr C}(X_0^*,X_1^*)^{(n)}_c$. Pick $\e>0$ and a function $g:\St_{a,b}\To \Sigma^n$ such that $g(c)=x$ with $\|g\|\leq (1+\e)\|x\|_{Y_c}$. Now, pick $h\in \mathscr C(X_0^*,X_1^*)$ such that $\tau_{(n,0]}h(c)=\xi$, with $\|h\|\leq (1+\e)\|\xi\|$.
Since $X_0^*\cap X_1^*=\Sigma^*$, slightly perturbing $\xi$ if necessary, we may assume that $h$ has the form (\ref{petunin}), with vectors in $\Sigma^*$. Then the components of  $\tau_{(n,0]}h$ are $\Sigma^*$-bounded on $\St_{a,b}$ and since $g$ is $\Sigma^n$-bounded the function
$$
f(z)= T_n(\tau_{(n,0]}h(z))(g(z))
$$
is bounded analytic on $\St_{a,b}$ and $f(c)=T_n\xi(x)$. But, for $z\in\partial\St_{a,b}$ one has
$$
|f(z)|\leq \big{\|}T_n:{\mathscr C}(X_0^*, X_1^*)^{(n)}_{z}\To\big{(}X^{(n)}_{z}\big{)}^*\big{\|}\:\big{\|}\tau_{(n,0]}h(z)\big{\|}_{{\mathscr C}(X_0^*, X_1^*)^{(n)}_{z}}
\:\big{\|}g(z)\big{\|}_{X^{(n)}_z}\leq
\frac{(1+\e)^2\big{\|}\xi\big{\|}_{{\mathscr C}(X_0^*, X_1^*)^{(n)}_{c}}\large{\|}x{\|}_{Y_c}}{\min(a,1-b)^{n-1}}
$$
since for $z\in\partial\St_{a,b}$ the space ${X^{(n)}_z}$ agrees with $Y_a$ when $\Re(z)=a$ and with $Y_b$ when $\Re(z)=b$. The result follows from the maximum principle. The second part is clear: if $\Delta^n$ is dense in $X^{(n)}_a\cap X^{(n)}_b$, then it is dense in $[X^{(n)}_a, X^{(n)}_b]_c$ too.\end{proof}

One may wonder if the irritating hypothesis about the density of $\Delta^n$ in $X^{(n)}_a\cap X^{(n)}_b$ is really necessary to get the identity $[X^{(n)}_a, X^{(n)}_b]_c= X^{(n)}_c$. Also, if $(X_0,X_1)$ is a regular Banach couple with intersection $\Delta$ and $0<a<b<1$, is $\Delta^2$ always dense in $X^{(2)}_a\cap X^{(2)}_b$?

The reader may observe that no acceptable space has been used. The question of which admissible space could have been, and could now be, used to obtain the higher order Rochberg spaces admits several answers. The most obvious is to choose:
$$
\mathscr D=\left\{g\in \mathscr C\big{(}X^{(n)}_a, X^{(n)}_b\big{)}: g(z)\in X^{(n)}_z \text{ for } a\leq \Re(z)\leq b\right\}.
$$

One has:

\begin{corollary}\label{cor:looming} With the same notations as above, $\mathscr D$ is an admissible space of analytic functions on the strip $\St_{a,b}$ and for each $z\in \St_{a,b}$, one has $\mathscr D_z=X^{(n)}_z$, with equivalent norms. Besides, if $x\in X^{(n)}_z$ and $f\in \CX$ is  such that $x=\tau_{(n,0]}f(z)$ and $\|f\|_{\CX}\approx \|x\|_{X^{(n)}_c}$, then, if
 $F$ is the restriction of $\tau_{(n,0]}f$ to $\St_{a,b}$, one has $F(z)=x$, and $\|F\|_{\mathscr D}=\|f\|_{\CX}\leq C\|x\|_{\mathscr D_z}$, where $C$ is a constant depending on $z$, but not on $x$.\end{corollary}

\begin{proof} To prove that $\mathscr D$ is admissible it suffices to check that if $\varphi:\St_{a,b}\To\mathbb D$ is a conformal equivalence, $g:\St_{a,b}\To \Sigma^n$ is analytic  and $\varphi g\in\mathscr D$, then $g\in\mathscr D$. Of course that $g\in \mathscr C\big{(}X^{(n)}_a, X^{(n)}_b\big{)}$. Let us see that $g(z)\in X^{(n)}_z$ for all $z\in\St_{a,b}$. This is obvious if $\varphi(z)\neq 0$.
Put $\zeta=\varphi^{-1}(0)$ and notice that the reasoning about $R$ contained in the proof of Theorem~\ref{prop:couples} shows that the restriction of $g $ to the line $\Re(z)=\Re(\zeta)$ is a continuous map with values in  $[X^{(n)}_a, X^{(n)}_b]_{\Re\zeta}$. As $g(z)$ belongs to $X^{(n)}_{z}=X^{(n)}_{\Re\zeta}$ for every $z\neq \zeta$ in the line $\Re(z)=\Re(\zeta)$ and this space is closed in $[X^{(n)}_a, X^{(n)}_b]_{\Re\zeta}$, we conclude that $g(\zeta)\in X^{(n)}_{\Re\zeta}$ and so $\mathscr D$ is admissible. The ``besides'' part is clear after Theorem~\ref{prop:couples}.
\end{proof}

Thus, starting with a Banach couple $(X_0,X_1)$ sitting on $\St$ one obtains the family $X_c=\CX_c$ and the corresponding Rochberg spaces $X^{(n)}_c$ for $0<c<1$. These spaces can be twisted in two ways: one is forming the space $X^{(2n)}_c$ which leads to the self-extension
\begin{equation}\label{eq:2n}
\begin{CD}
0@>>> X^{(n)}_c@>>> X^{(2n)}_c @>>> X^{(n)}_c @>>> 0
\end{CD}
\end{equation}
described in Section~\ref{sec:entwining}. But the preceding Corollary \ref{cor:looming} also opens up the possibility of considering $X^{(n)}_c$ as one of the spaces of the analytic family induced by $\mathscr D$ which leads to the self-extension
\begin{equation}\label{eq:D2}
\begin{CD}
0@>>> X^{(n)}_c@>>> \mathscr D^{(2)}_c @>>> X^{(n)}_c @>>> 0
\end{CD}
\end{equation}
These extensions are different. Indeed, the differential  associated to (\ref{eq:2n}) is obtained as follows: given $x=(x_{n-1},\dots,x_0)$ in $X^{(n)}_c$ we select $f\in\CX$ such that $x=\tau_{(n,0]}f(c)$, with $\|f\|_{\CX}\approx \|x\|_{X^{(n)}_c}$ and set
$$
\Omega^{n,n}(x)= \tau_{(2n,n]}f(c).
$$
As for (\ref{eq:D2}) we can use the restriction $F$ of $\tau_{(n,0]}f$ to $\St_{a,b}$ as an extremal for $x$ in $\mathscr D$, so that the corresponding derivation is
$$
\Phi^{1,1}(x)=F'(c)=\left(\frac{f^{(n)}(c)}{(n-1)!}, \frac{f^{(n-1)}(c)}{(n-2)!},\dots, f'(c)  \right)=\Bigg{(}  \underbrace{n\frac{f^{(n)}(c)}{n!}}_{\text{nonlinear}}, \underbrace{(n-1)x_{n-1},\dots, x_1}_{\text{linear part}}  \Bigg{)}.
$$
This seems to indicate that, in a sense, (\ref{eq:2n}) ``twists'' $X^{(n)}_c$ more than (\ref{eq:D2}) does. This point will be discussed in depth in Section~\ref{sec:rambling}, in the broader context of acceptable spaces.

\subsection{The issue of families} To explain the role of acceptable families, let us explain why we have encountered insurmountable difficulties to generalize Theorem~\ref{prop:couples} to admissible families. Let $\U$ be a domain and let $\V$ be a subdomain with compact closure contained in $\U$. We fix conformal equivalences $\varphi:\D\To\U$ and $\phi:\D\To\V$ having the extension properties required in Section~\ref{sec:int-fam} and we denote again by $\varphi$ and $\phi$ their extensions to $\T$. These are well-defined up to a null set. \medskip

Suppose we are given an admissible interpolation family on $\U$, say $\mathcal X = (X_z)_{z \in \partial \U}$, with ambient space $\Sigma$, intersection $\Delta$ and containing function $k$. Fixing $n\geq 1$ we can consider the family of Rochberg spaces $X^{(n)}_z$ with $z$ varying in $\U$ (note that there are no Rochberg spaces on the original boundary $\partial\U!)$, which includes $\partial\V$. In this way we obtain another family, parametrized by $\partial\V$, namely $\mathcal{Y}=(Y_n)_{v\in\partial\V}$, where $Y_v=X^{(n)}_v$ for $v\in\partial\V$. We would like to make $\mathcal{Y}$ an interpolation family. To this end we can choose $\Sigma^n$ as the ambient space and $\Delta^n$ as the intersection space of $\mathcal Y$ so that the compactness of $\overline{\V}$ resolves the ``containing function'' issue:

\begin{lemma}
Under the above hypotheses there is a constant $C$ such that if $(x_{n-1},\dots,x_0)$ belongs to $X^{(n)}_v$ for some $v\in  \overline{\V}$, then $\sum_{0\leq i<n} \|x_i\|_{\Sigma} \leq C\|(x_{n-1},\dots,x_0)\|_{X^{(n)}_v}$.
\end{lemma}

\begin{proof}
Let $k:\partial\U\To(0,\infty)$ be the containing function of $\mathcal X$ and  $K:\D\To\mathbb C$ be the outer function associated to $k\circ\varphi$. Then, for every $u\in\U$, every $n\geq 0$ and every $R<\dist(u,\partial\U)$, one has
$$
\|\delta_u^{(n)}:\F(\mathcal{X})\To \Sigma\|\leq \frac{n!M(u,R)}{R^n},\quad\quad\text{where}\quad\quad M(u,R)=\max_{|u-z|\leq R}|K(\varphi^{-1}(z)|.
$$
This is straightforward from Cauchy's estimates.
Let $r={1\over 2}\dist(\V,\partial\U)$. Then
$\overline{\V}_r=\overline{\V}+\overline{\D}_r=\{v+z: v\in\overline{\V}, |z|\leq r\}
$
is a compact subset of $\U$ containing $\overline{\V}$, where $K\circ\varphi^{-1}$ has to be bounded, say by $M$.
Thus, for every $v\in\overline{\V}$, in particular for $v\in \partial\V$ one has $\|\delta_v^{(n)}:\F(\mathcal{X})\To \Sigma\|\leq n!M/r^n$.

Now, pick $v\in\partial\V$ and $x=(x_{n-1},\dots,x_0)$ in $X_v^{(n)}$. If $f\in\F(\mathcal{X})$  is such that $\tau_{(n,0]}f(v)=x$, we have
\[
\sum_{0\leq i<n} \|x_i\|_{\Sigma} \leq M\left(\sum_{0\leq i<n}r^{-i}\right) \|f\|_\F,
\]
as required.\end{proof}
This shows that $\Sigma^n$, with the sum norm, is a containing space for the family $\mathcal Y$, with containing function (actually constant) $M\left(\sum_{0\leq i<n}r^{-i}\right)$. Up to here the good news. The bad news are that we have been unable to establish the measurability of the function $v\in\partial\V\longmapsto\|x\|_{X_v^{(n)}}$ for fixed $x\in\Delta^n$, that is, we cannot guarantee that $\mathcal Y$ is an interpolation family. In the case of couples this was automatic as these functions are constant on each vertical line! Worse yet, even if one could stablish measurability in some cases (e.g., if the extremals are unique) or if one could dispose of this issue (replacing $N^+_\V$ by $A_\V$, or  something like that), it is unclear whether the hypothesized interpolation family would be admissible. All we know is the following result, which obviates these difficulties adding to the hypothesis a statement that we would have liked to put into the thesis, namely that the family of derived spaces is admissible.

\begin{proposition}\label{interpofamily}
With the above notations, if $\mathcal Y$ is an admissible interpolation family with intersection space $\Delta^n$, then,
 for every $z\in \mathbb V$, one has $Y_z=X_z^{(n)}$ with equivalence of norms.
\end{proposition}

\begin{proof}
Let us prove first that, for each $v\in\V$, one has $X^{(n)}_v\subset Y_v$, and the inclusion is contractive.
Pick $x\in \Delta^n$ and then $g\in\G(\mathcal{X})$ such that $x=\tau_{(n,0]}g(v)$. 
Let $f:\V\To \Sigma^n$ be the restriction of
$\tau_{(n,0]}g$ to $\V$. Then $f\in\G(\mathcal Y)$: indeed, if we write
$g = \sum g_j a_j$, with $g_j\in N^+_\U$ and $a_j\in\Delta$, then the successive derivatives of each $g_j$ are all bounded on $\V$ and so they belong to $ N^+_\V$.
Besides, we have $\|f(z)\|_{X^{(n)}_z}\leq \|g\|_\F$ for every $z\in\partial\V$, so we have $f\in\G(\mathcal Y)$, with $x=\delta_v f$, and
$$\|x\|_{Y_v}\leq \|f\|_{\G(\mathcal Y)}\leq \|g\|_{\F(\mathcal X)}.$$
Since $g$ is arbitrary and $\Delta^n$ is dense in $X^{(n)}_v$ we are done.\medskip

We now prove the reversed containment and obtain the corresponding bound.
This part uses  duality in a critical way. First, since $T_n: {\W(\mathcal X)^{(n)}_v}\To (X_v^{(n)})^*$ is an isomorphism, it suffices to see that there is a constant $K$ such that, if $x\in\Delta^n$ and $\xi\in (\Delta^\star)^n, \|\xi\|_{\W(\mathcal X)^{(n)}_v}<1$, then
$$
|T_n\xi(x)|\leq K \|x\|_{Y_v}.
$$
So, take $g\in\G(\mathcal Y)$ such that $g(v)=x$, with $\|g\|_{\F(\mathcal Y)}\leq (1+\e)\|x\|_{Y_v}$ and $h\in\WX$ so that $\xi=\tau_{(n,0]}h(v)$, with $\|h\|_{\WX}\leq 1$.

By \cite[Proposition 2.5]{Coifman1982} we can assume that the coefficient functions of $g$ are bounded on $\V$. Therefore, using the conformal map $\phi:\D\to\V$ we may consider the function $f:\D\To\mathbb C$ defined by
$$f(z)=(T_n (\tau_{(n,0]}h(\phi(z))))(g(\phi(z))).$$
Then $f$ is analytic, bounded on $\D$ and $f(\phi^{-1}(v))=T_n\xi(x)$.
Moreover, for almost every $z\in\T$, one has
$$
|f(z)|\leq \big{\|}T_n:\W^{(n)}_{\phi(z)}\To\big{(}X^{(n)}_{\phi(z)}\big{)}^*\big{\|}\:\big{\|}\tau_{(n,0]}h(\phi(z))\big{\|}_{\W^{(n)}_{\phi(z)}}
\:\big{\|}g(\phi(z))\big{\|}_{X^{(n)}_{\phi(z)}}\leq\frac{\|g(\phi(z))\|_{Y_{\phi(z)}}}{\dist(\partial\V,\partial\U)^{n-1}}\leq
\frac{(1+\e)\|x\|_{Y_v}}{\dist(\partial\V,\partial\U)^{n-1}},
$$
and the result follows from the maximum principle.
\end{proof}

\subsection{The case of analytic families on the disc}

This and the next sections do what we wanted to do in the previous section at the cost of working in the general setting of acceptable spaces.
Precisely, what we will show is that if $\F$ is an acceptable space of analytic functions on a domain $\U$ then the family of Rochberg spaces $\F^{(n)}_z$, for $z$ varying in $\U$ and $n\geq 2$ fixed, is the analytic family associated to another acceptable space which is naturally attached to $\F$. This result has no counterpart for admissible spaces. It actually was our original motivation to introduce the notion of an acceptable space and what fully justifies our approach. We will treat in this section the case where the domain is the disc, taking advantage of the fact that the underlying algebra $A^\infty$ admits differentiation. The adjustments required to work on general domains are carried out in the next section.\medskip

Let $\F$ be an acceptable space on the disc and let $\mathscr H=\mathscr H(\D,\Sigma)$ be the space of all holomorphic functions from $\D$ to $\Sigma$, the ambient space of $\F$. We inductively define a sequence of Banach spaces $\Fn$, formally subspaces of the product $\mathscr H^n$  as follows:\medskip

$\bullet$ $\F^{(1)} = \F$.

$\bullet$ Once $\Fn$ is defined we consider the linear map $\tau_{[n,1]}:\F\longrightarrow\mathscr H^n$ and set
$$
\F^{(n+1)}=\Fn\oplus_{\tau_{[n,1]}}\F=
\left\{(f_n,\dots,f_1,f)\in\mathscr H^{n+1}:
f\in\F \text{ and }
(f_n,\dots,f_1)-\tau_{[n,1]}(f)\in\Fn\right\},
$$
endowed with the norm $
\|(f_n,\dots,f_1,f)\|_{\mathscr F^{(n+1)}}=
\left\|
(f_n,\dots,f_1)-\tau_{[n,1]}(f)\right\|_{\Fn}+\|f\|_\F.
$\medskip

Observe that $\F^{(2)}$ consist of those pairs $(g,f)$ such that both $f$ and $g-f'$ are in $\F$, with norm $\|(g,f)\|=\|g-f'\|+\|f\|$. To compute $\F^{(3)}$, pick $(f_2,f_1,f_0)$. Of course $f_0$ has to be in $\F$, while
$(f_2-f_0''/2,f_1-f_0')$ must be in $\F^{(2)}$, that is, both
$f_1-f_0'$ and $f_2-f_0''/2-(f_1-f_0')'$ must be in $\F$, so in the end the norm of $(f_2,f_1,f_0)$ in $\F^{(3)}$ is
$\|f_2-f_1'+f_0''-f_0''/2\|+\|f_1-f_0'\|+\|f_0\|$. Instead of spoiling all the fun presenting the 4D case, let us see an explicit formula that works in general. The form of the coefficients that appear in the following result can somehow be considered a lucky strike:

\begin{lemma}\label{lem:stroke}
Fix $n\geq 1$ and let $f_i\in \mathscr H$ for $0\leq i<n$. Then $(f_{n-1},\dots,f_0)$ belongs to $\F^{(n)}$ if and only if for each $0\leq i<n$ the sum
$$
f_i+\sum_{1\leq k\leq i}\frac{(-1)^k}{k!}f_{i-k}^{(k)}
$$
falls into $\F$, where the sum over the empty set is treated as zero. Moreover, for such an array $(f_{n-1},\dots,f_0)$ one has
$$
\|(f_{n-1},\dots,f_0)\|_{\Fn}= \|f_0\|_\F+ \sum_{0<i<n}\left\| f_i+\sum_{1\leq k\leq i}\frac{(-1)^k}{k!}f_{i-k}^{(k)}	\right\|_\F.
$$
\end{lemma}

\begin{proof}The proof goes by induction on $n$. The initial step $n=1$ is trivial, so let us assume that the lemma holds for $n$ and let us check the corresponding statement for $n+1$. Pick $n+1$ functions $f_i\in\mathscr H$ for $0\leq i\leq n$. By the very definition,
$(f_{n},\dots,f_0)\in\F^{(n+1)}$ if and only if $f_0\in\F$ and $(f_{n},\dots,f_1)-\tau_{[n,1]}f_0$ belongs to $\Fn$. Write
$$
(f_{n},\dots,f_1)-\tau_{[n,1]}f_0=\left(f_n-\frac{f_0^{(n)}}{n!},\dots, f_1- f_0'  \right) =  (g_{n-1},\dots,g_0).
$$
Then the induction hypothesis says that  $(g_{n-1},\dots,g_0)\in\Fn$ if and only if for each $0\leq i\leq n-1$ the following sum belongs to $\F$:
$$
g_i+\sum_{0< k\leq i}\frac{(-1)^k}{k!}g_{i-k}^{(k)}
=
f_{i+1}- \frac{f_0^{(i+1)}}{(i+1)!}+\sum_{0< k\leq i}\frac{(-1)^k}{k!}\left(f_{i+1-k}^{(k)}-\frac{f_0^{(i+1-k+k)}}{(i+1)!}\right) = f_{i+1}+
\sum_{0< k\leq i+1}\frac{(-1)^k}{k!}f_{i+1-k}^{(k)}
$$
because
$$
\frac{-1}{(i+1)!}+ \sum_{0< k\leq i}\frac{(-1)^k}{k!}\frac{-1}{(i+1)!}= \frac{(-1)^{i+1}}{(i+1)!};
$$
(see Equation~\ref{eq:1/n!}). Probably it is not necessary to say anything more.
\end{proof}

Note that the Lemma  implies, among other things, that $$\F^{(n+1)}=\F\oplus_\Phi\Fn,\quad\quad\text{ with}\quad\quad
\Phi(f_{n-1},\dots,f_1,f_0)=-\sum_{1\leq k\leq n}\frac{(-1)^{k}}{k!}f_{n-k}^{(k)}=-\sum_{0\leq k\leq n-1}\frac{(-1)^{n-k}}{(n-k)!}f_{k}^{(n-k)}
$$
and also:

\begin{corollary}\label{cor:and-also} With the same notations as before  $F\in \mathscr H^n$ belongs to $\F^{(n)}$ if and only it has the form
$$
F=\left(  \frac{f_0^{(n-1)}}{(n-1)!}+\frac{f_1^{(n-2)}}{(n-2)!}+\dots+f_{n-1},\,\dots\, ,  \frac{f_0''}{2!}+f'_1+f_2\,,\, f_0'+f_1\,,\,f_0\right)
$$
with $f_i\in\F$ for $0\leq i<n$, in which case $\|F\|_{\F^{(n)}}$ is equivalent to $\sum_{0\leq i<n}\|f_i\|_{\F}$.
\end{corollary}

Let us then prove what has brought us here:

\begin{proposition}\label{prop:Fnaccep}
If $\F$ is an acceptable space of analytic functions on the disc, then so is $\Fn$ for every $n\geq 1$. Moreover:
\begin{itemize}
\item If $f\in\F$, then $\tau_{(n,0]}f\in\Fn$ and $\|\tau_{(n,0]}f\|_{\Fn}=\|f\|_\F$.
\item The analytic family associated to $\Fn$ are the Rochberg spaces $(\Fn_z)_{z\in\D}$, up to equivalence of norms.
\end{itemize}
\end{proposition}

\begin{proof} We first observe that each $n$-tuple $(g_n,\dots,g_1)$ in $\prod_{i=1}^n\mathscr H(\D,\Sigma)$ can be seen as an analytic function from $\D$ to $\Sigma^n$ just letting
$(g_n,\dots,g_1)(z)=(g_n(z),\dots,g_1(z))$, where $\Sigma^n$ can be equipped with the direct sum norm, so certainly $\Fn$ is a space of analytic functions.\medskip

The result is trivial when $n=1$ and will be established by induction on $n$. So, let us assume it true for $1,\dots, n$ and prove it for $n+1$.
To check completeness, just observe that $\F^{(n+1)}$ is a twisted sum of $\F$ by $\F^{(n)}$ and that those spaces are complete by the induction hypothesis. A classical 3-space result \cite{castgonz} then asserts that a twisted sum of complete spaces is complete. In  order to prove that the evaluations $\delta_z: \F^{(n+1)}\longrightarrow \Sigma^n$ are bounded we can assume that  $\delta_z: \F^{(n)}\longrightarrow \Sigma^{n}$ are bounded. As explained in Section~\ref{sec:entwining}, the successive derivatives $\delta_z^{(k)}:\F\longrightarrow \Sigma$ are all bounded. Pick $(f_n,\dots,f_1,f_0)\in\F^{(n+1)}$ and consider the decomposition
$$
(f_n,\dots,f_1, f_0)=  (f_n,\dots,f_1, f_0)-\tau_{[n,0]}f_0+ \tau_{[n,0]}f_0
$$
We have
$$
\|(f_n,\dots,f_0)\|_{\F^{(n+1)}}= \|(f_n,\dots,f_1)-\tau_{[n,1]}f_0\|_{\Fn}+\|f_0\|_\F.
$$
Also,
\begin{align*}
\|\delta_z(f_n,\dots,f_0)\|_{\Sigma^{n+1}}&\leq \| \delta_z\big{(}(f_n,\dots,f_1)-\tau_{[n,1]}f_0 \big{)}  \|_{\Sigma^n}+ \sum_{0\leq k\leq n} \left\|\frac{f^{(k)}_0(z)}{k!}\right\|_\Sigma\\
&\leq \left\|\delta_z:\F^{(n)}\to \Sigma^{n}\right\|\|(f_n,\dots,f_1)-\tau_{[n,1]}f_0\|_{\Fn} +
\sum_{0\leq k\leq n}  \left\|\frac{\delta^{(k)}}{k!}:\F\to \Sigma\right\|\|f\|_\F,
\end{align*}
which is enough.

Let us check that $\F^{(n+1)}$ is an $A^\infty$-module under pointwise multiplication assuming that so is $\F^{(n)}$.
As a preparation we consider the following general situation. Suppose we have a (topological) algebra $A$ and that $X$ and $Y$ are topological left-modules over $A$. Let $H$ be another $A$-module, not necessarily carrying a topology, that contains $Y$ as a submodule. Finally, suppose $\Phi:X\to H$ is quasilinear from $X$ to $Y$; see Definition~\ref{def:quasilinear}.
 It is very easy to see that
 the ``coordinatewise'' product $a(h,x)=(ah,ax)$ makes $
Y\oplus_\Phi X$ into a topological $A$-module if and only if for every $a\in A$ and $x\in X$ one has $\Phi(ax)-a\Phi(x)\in Y$ and
$$
\|\Phi(ax)-a\Phi(x)\|\To 0\quad\quad\text{as}\quad\quad(a,x)\To 0 \text{ in } A\times X.
$$
As the space $\F^{(n+1)}$ is just  $\F^{(n)}\oplus_\Phi\F$ when $\Phi$ is the quasilinear map (linear in fact) given by $\tau_{[n,1]}:\F\To \mathscr H^{n}$ what we need to prove is that if $f\in\F$ and $a\in A^\infty$, then the difference
$\tau_{[n,1]}(af)-a\tau_{[n,1]}(f)$ falls into $\Fn$ and
\begin{equation}
\|\tau_{[n,1]}(af)-a\tau_{[n,1]}(f)\|_{\Fn}\To 0\quad\quad\text{as}\quad\quad(a,f)\To 0\text{ in } A^\infty\times\F.
\end{equation}
Note that if $f\in\F$, then, for each $k\geq 1$, the array $
\tau_{[k,0]}(f)$ belongs to $\F^{(k+1)}$, with $\|\tau_{[k,0]}(f)\|_{\F^{(k+1)}}=\|f\|_{\F}$ and so every array of the form
$$ \left(\frac{f^{(k)}}{k!},\dots,f', f, 0,\dots,0 \right),
$$
ending with $\ell$ zeroes, belongs to $\F^{(k+\ell+1)}$ and its norm there agrees with $\|f\|_\F$. Fix now $f\in\F, a\in  A^\infty$ and let us compute the difference $\tau_{[n,1]}(af)- a\tau_{[n,1]}(f)$. Note, that, by the Leibniz formula
$$
\frac{(af)^{(k)}}{k!}= \sum_{0\leq i\leq k} \frac{a^{(k-i)}}{(k-i)!}\frac{f^{(i)}}{i!},
$$
so
\begin{align*}
\tau_{[n,1]}(af)&=
\left(\frac{(af)^{(n)}}{n!},\dots,(af)'\right)\\
&=
\underbrace{a\left(\frac{f^{(n)}}{n!},\dots,f'\right)}_{a\tau_{[n,1]}f}+a'\left(\frac{f^{(n-1)}}{(n-1)!},\dots,f, \right)+\frac{a''}{2!}\left(\frac{f^{(n-2)}}{(n-2)!},\dots,0 \right)+\dots+\frac{a^{(n)}}{n!}\left(f,0, \dots, 0 \right)
\end{align*}
Hence
$$
\tau_{[n,1]}(af)-a\tau_{[n,1]}(f)=
a'\left(\frac{f^{(n-1)}}{(n-1)!},\dots,f\right)+\frac{a''}{2!}\left(\frac{f^{(n-2)}}{(n-2)!},\dots,0\right)+\dots+\frac{a^{(n)}}{n!}\left(f,0, \dots, 0 \right),
$$
with each summand in $\Fn$,
and
$$
\|\tau_{[n,1]}(af)-a\tau_{[n,1]}(f)\|_{\Fn}\leq \sum_{1\leq k\leq n}\left\|\frac{a^{(k)}}{k!}\right\|_{L(\F^{(n-k)})}\|f\|_\F.
$$
To complete the proof that $\F^{(n+1)}$ is acceptable  let us assume that  $(f_n,\dots,f_0)\in\mathscr H(\D, \Sigma^{n+1})$ and $\phi\in\Aut(\D)$ are such that $\phi (f_n,\dots,f_0)$ falls into  $\F^{(n+1)}$. We must check that  $(f_n,\dots,f_1,f_0)$ belongs to  $\F^{(n+1)}$ and that
$$
\|(f_n,\dots,f_1,f_0)\|_{\F^{(n+1)}}\leq K[\phi,n+1] \|\phi(f_n,\dots,f_1,f_0)\|_{\F^{(n+1)}},
$$
where $K[\phi,n+1]$ is a constant depending on $\phi$ and the ``dimension'' only. The hypothesis means that $\phi f_0\in\F$ (hence $f_0\in\F$) and $\phi(f_n,\dots,f_1)-\tau_{[n,1]}(\phi f_0)\in\Fn$.
On the other hand, since $\phi\in A^\infty$ (see the Appendix), we know from the previous step that the difference
$\tau_{[n,1]}(\phi f_0)- \phi\tau_{[n,1]}(f_0)$ belongs to $\Fn$. Thus,
$$
\phi(f_n,\dots,f_1)-\phi\tau_{[n,1]}(f_0)\in\Fn
$$
and the induction step yields $(f_n,\dots,f_1)-\tau_{[n,1]}(f_0)\in\Fn$, hence
$(f_n,\dots,f_1, f_0)\in\F^{(n+1)}$.\medskip

As for the norm, one has
\begin{align*}
\|(f_n,\dots,f_1, & f_0)\|_{\F^{(n+1)}}=\|(f_n,\dots,f_1)-\tau_{[n,1]}(f_0)\|_{\Fn}+\|f_0\|_\F\\
&\leq K[\phi,n] \|\phi(f_n,\dots,f_1)-\phi\tau_{[n,1]}(f_0)\|_{\Fn}+K[\phi,1]\|\phi f_0\|_\F\\
&\leq K[\phi,n] \left(\|\phi(f_n,\dots,f_1)-\tau_{[n,1]}(\phi f_0)\|_{\Fn} + \|\tau_{[n,1]}(\phi f_0)- \phi\tau_{[n,1]}(f_0)\|_{\Fn} \right)+K[\phi,1]\|\phi f_0\|_\F\\
&\leq \max(K[\phi,n],K[\phi,1])\|\phi(f_n,\dots,f_1,f_0)\|_{{\F^{(n+1)}}}+ K[\phi, n]
\sum_{1\leq k\leq n}\left\|\frac{\phi^{(k)}}{k!}\right\|_{L(\F^{(n-k)})}\|f_0\|_\F,
\end{align*}
which is enough as it implies that
$$
K[\phi,n+1]\leq \max\Big{(}K[\phi,n],K[\phi,1]\Big{)}+ K[\phi,n]K[\phi,1]\sum_{1\leq k\leq n}\left\|\frac{\phi^{(k)}}{k!}\right\|_{L(\F^{(n-k)})}.
$$

Finally, we prove the ``moreover'' part. For each $k\geq 1$ let $(\F^{(k)})_z$ denote the analytic family induced by $\F^{(k)}$, while
we keep the notation $\F^{(k)}_z$ for the $k$-th Rochberg space induced by $\F$ at $z$. In particular:
\begin{align*}
(\F^{(n+1)})_z&=\{x\in \Sigma^{n+1}: x=F(z) \text { for some } F\in \F^{(n+1)} \};\\
\F^{(n+1)}_z&=\{x\in \Sigma^{n+1}: x=\tau_{[n,0]}f(z) \text { for some } f\in \F\}.
\end{align*}

Now, if $f\in\F$, then the array $F=\tau_{[n,0]}(f)$ belongs to $\F^{(n+1)}$ by the very definition, and evaluating at $z$ one obtains the Taylor coefficients of $f$. Besides, $\|\tau_{[n,0]}(f)\|_{\F^{(n+1)}}=\|f\|_{\F}$, hence $(\F^{(n+1)})_z$ contains $\F^{(n+1)}_z$ and the inclusion is contractive. To establish the other containment, one has to check that if $(f_n,\dots,f_0)$ belongs to $\F^{(n+1)}$ then, for each $z\in\D$, there is $f\in \F$ such that
$$
f_k(z)=\frac{f^{(k)}(z)}{k!}\quad\quad(0\leq k\leq n)
$$
with $\|f\|_\F\leq M\|(f_n,\dots,f_0)\|_{\F^{(n+1)}}$, where $M=M[z,n+1]$ depends only on the dimension and on $z$, but not on the array. So, fix $z\in\D$ and pick $(f_n,\dots,f_0)$ in $\F^{(n+1)}$.
Then since the array
$(f_n,\dots,f_1)-\tau_{[n,1]}f_0$ belongs to $\Fn$ we can assume by the induction hypothesis that there is $g\in\F$ such that
\begin{equation}\label{eq:gz}
g(z)=f_1(z)-f_0'(z), \dots, \frac{g^{(n-1)}(z)}{(n-1)!}=f_n(z)- \frac{f_0^{(n)}(z)}{n!},
\end{equation}
with $\|g\|_\F\leq M[z,n]\left\|(f_n,\dots,f_1)-\tau_{[n,1]}f_0\right\|_{\Fn}$.
Take $\phi\in\Aut(\D)$ vanishing at $z$ and use \cite[Lemma~1]{cck} to get a polynomial $P$ of degree at most $n$  so that if $a=P(\phi)$,
then $a^{(k)}(z)=\delta_{k1}$ (Kronecker delta) for $0\leq k\leq n$. Obviously, $a\in A^\infty$ and so $f=ag+f_0\in\F$. We have
\begin{align*}
\|f\|_\F&\leq \|a\|_{L(\F)}\|g\|_\F+\|f_0\|_\F\\
&\leq
 \|a\|_{L(\F)}  M[z,n]\left\|\left(f_n,\dots,f_1\right)-\tau_{[n,1]}f_0\right\|_{\Fn}+\|f_0\|_\F\\
 &\leq  \max\left( \|a\|_{L(\F)}  M[z,n], 1\right)\|(f_n,\dots,f_0)\|_{\F^{(n+1)}}.
\end{align*}
As for the Taylor coefficients, by Leibniz rule and (\ref{eq:gz}),
$$
\frac{f^{(k)}(z)}{k!}=\frac{f_0^{(k)}{(z)}}{k!}+ \sum_{0\leq i\leq k}\frac{a^{(i)}(z)}{i!}
\frac{g^{(k-i)}(z)}{(k-i)!}=
\frac{f_0^{(k)}{(z)}}{k!}+  \frac{g^{(k-1)}(z)}{(k-1)!}=f_k(z).\qedhere
$$
\end{proof}

\subsection{General domains}\label{sec:general} We transplant our results from the disc to general domains.
The main obstruction to proceed as we did in Proposition~\ref{prop:Fnaccep} is that
the grafted algebras $\AU$ are not closed under differentiation, even if $\U$ is a strip (see the Appendix). Therefore, most of the computations done along points 5 and 6 of that proof just do not make any sense for general domains. The idea is then to use a conformal map between $\U$ and $\D$ to transfer the acceptable space $\F$ from $\U$ to $\D$, then use Proposition~\ref{prop:Fnaccep} and then move back to $\U$. This involves the most basic operations in calculus: Chain and Leibniz rule. The paper \cite{rochtrans} contains much deeper ``translations'' to vector valued analytic functions of much deeper facts about complex analytic functions.

\subsubsection{Chain rule}\label{sec:chainrule}
Let $\F$ be an acceptable space on $\U$ and suppose $\psi:\mathbb V\to\U$ is a conformal equivalence. Then we can consider the space
$$
\mathscr G=\psi^*[\F]=\{g\in\mathscr H(\mathbb V,\Sigma): g=f\circ\psi, f\in\F\},
$$
with norm $\|g\|_{\mathscr G}=\|f\|_\F$. It is clear that $\G$ is acceptable, or admissible if $\F$ is. In some sense, $\G$ and $\F$ are ``equivalent'' objects. This is indeed the case for the ``degree zero'' theory as shown by the fact that, for each $z\in\mathbb V$, one has $\mathscr G_z=\F_{\psi(z)}$, with identical norms. We omit the obvious proof.\medskip

What about the corresponding Rochberg spaces? They are still isometric but, in general, different. To see this, fix $z\in\mathbb V$ and put $u=\psi(z)$. Take $(x_1,x_0)\in \F_u^{(2)}$ and pick $f\in\F$ so that $x_1=f'(u), x_0=f(u)$. Then take $g=f\circ\psi$ and evaluate $\tau_{[1,0]}g$ at $z$:
$$(g'(z),g(z))=(f'(u)\psi'(z), f(u))= (\psi'(z)x_1,x_0).$$
This shows at once:
\begin{itemize}
\item
The map $(x_1,x_0)\mapsto  (\psi'(z)x_1,x_0)$ is a surjective isometry between  $ \F_u^{(2)}$  and $ \G_z^{(2)}$.
\item If $\psi'(z)\neq 1$, then $\F_u^{(2)}=\G_z^{(2)}$ as subspaces of $\Sigma^2$ if and only if $\F_u^{(2)}=\F_u\times\F_u$.
\item It $\lambda=\psi'(z)$, then we have a commutative diagram (recall that $\F_u$ and $\G_z$ are the same space)
$$
\begin{CD}
0 @>>> \F_u @>>> \F_u^{(2)} @>>> \F_u@>>> 0\\
 & & @V\lambda VV @V\lambda\times{\bf 1}VV  @| \\
0@>>> \G_z @>>> \G_z^{(2)} @>>> \G_z @>>> 0
\end{CD}
$$
in which the middle arrow is an isometry.\medskip
\end{itemize}

In general we can describe nice isometries between  $\F_u^{(n)}$ and $\G_z^{(n)}$ as follows. Take $x\in \F_u^{(n)}$ and let $f\in \F$ be a representative, that is, $x=\tau_{(n,0]}f(u)$. Set $g=f\circ\psi$ and put $y=\tau_{(n,0]}g(z)$. It is clear that $y$ depends only on $x$ (if $f$ has a zero of order $k$ at $u$, then $g$ has a zero of order $k$ at $z$, and vice versa) and that this correspondence defines a surjective isometry between $\F_u^{(n)}$ and $\G_z^{(n)}$ that we may denote by $L[n,u]$ thus emphasizing the fact that it depends on the base point. To understand the dependence between the input $x=(x_{n-1},\dots,x_0)$ and the output
$y=(y_{n-1},\dots,y_0)$ we can invoke Fa\`a di Bruno's formula (see \cite{faa} for an exposition).
Write
$$
f(v)=\sum_{m\geq 0}x_m(v-u)^m\quad\quad\text{and}\quad\quad
\psi(w)=\sum_{m\geq 0}z_m(w-z)^m
$$
with positive radii of convergence. Then
$$
g(w)=f(\psi(w))=\sum_{m\geq 0}y_m(w-z)^m,
\quad\quad\text{
with
}\quad\quad
y_m=\sum_{(b_1,\dots,b_m)}\frac{z_1^{b_1}}{b_1 !}\cdots \frac{z_m^{b_m}}{b_m !}\: k!\:x_k,
$$
where the sum is taken over all different solutions $(b_1,\dots,b_m)$ of the equation
$b_1+2b_2+\cdots+mb_m=m$ in which each $b_i$ is a nonnegative integer and $k=b_1+b_2+\cdots+b_m$; in particular $k\leq m$. Hence, each $L[n,u]$ is implemented by an upper triangular matrix with complex coefficients that we will denote $\mathbf{FdB}$, with the understanding that $\mathbf{FdB}$ depends on $n, u$ and $\psi$.\medskip

Take $n,k\geq 1$ and let $\pi_{n+k,n}:\Sigma^{n+k}\To \Sigma^n$ denote the projection onto the last $n$ coordinates. Clearly, $L[n,u]\circ \pi_{n+k,n}= \pi_{n+k,n}\circ L[n+k,u]$, so $ L[n+k,u]$ maps the kernel of  $ \pi_{n+k,n}: \F_u^{(n+k)}\longrightarrow \F_u^{(n)}$ onto that of  $ \pi_{n+k,n}: \G_u^{(n+k)}\longrightarrow \G_u^{(n)}$ and we have a commutative diagram
\begin{equation}\label{dia:FuGz}
\begin{CD}
0 @>>> \F_u^{(k)} @>>> \F_u^{(n+k)} @>>> \F_u^{(n)}@>>> 0\\
 & & @V I VV @VL[n+k,u]VV  @VV L[n,u]  V\\
0@>>> \G_z^{(k)} @>>> \G_z^{(n+k)} @>>> \G_z^{(n)} @>>> 0
\end{CD}
\end{equation}
in which $I$ is an isomorphism, depending on $n, k$ and $u$, in general different from $L[k,u]$.\medskip

Moral: If you are interested in twisted sums, Banach space properties of the derived spaces and the like you can change variables without causing any harm to your conclusions. If you are rather interested in interpolation spaces, interpolation of operators and the like, you should be careful.

\subsubsection{Leibniz rule}
The preceding considerations suggest the following formal procedure to correct the distorsion introduced by a change of variable. Let $\F$ be an admissible/acceptable space of analytic functions from $\U$ to $\Sigma$ and suppose $L:\U\longrightarrow\Aut(\Sigma)$ is analytic when $\Aut(\Sigma)$ carries the restriction of the norm topology of $L(\Sigma)$. We can define a weighted version of $\F$, denoted $L_*[\F]$ with a slight abuse of notation, taking those functions $g:\U\longrightarrow \Sigma$ of the form $g(z)=L(z)(f(z))$, for some $f:\U\longrightarrow \Sigma$, with norm
 $
 \|g\|_{L_*[\F]}= \|f\|_\F
 $.
It is clear that  $L_*[\F]$ is admissible/acceptable if and only if $\F$ is. Moreover, for each $z\in\U$, one has ${L_*[\F]}_z=L(z)[\F_z]$ and that $L(z): \F_z\longrightarrow {L_*[\F]}_z$ is a surjective isometry.\medskip

 The connection between the Rochberg spaces of $\F$ and those of ${L_*[\F]}$ is as follows. Suppose $(x_{n-1},\dots, x_0)\in \Sigma^n$ belongs to $\Fn_z$ and that it agrees with the evaluation of $\tau_{[n-1,0]}(f)$ at $z$. Then $g(\zeta)=L(\zeta)(f(\zeta))$ belongs to $L_*[\F]$ and since by Leibniz's rule
 $$
 \frac{g^{(k)}(z)}{k!}=\sum_{0\leq i\leq k}\frac{L^{(k-i)}(z)}{(k-i)!}\left(\frac{f^{(i)}(z)}{i!}\right)
 $$
 we see that the isometry between $\Fn_z$ and ${L_*[\F]}_z^{(n)}$ is implemented by the following operator valued matrix
 evaluated at $z$
 $$
 \left( \begin{matrix}
\frac{L^{(n-1)}}{(n-1)!} & \frac{L^{(n-2)}}{(n-2)!}& \dots  & L' & L\\
 0 & \frac{L^{(n-2)}}{(n-2)!}& \dots & L'  & L\\
 \vdots & \vdots & \ddots &\vdots &\vdots\\
 0 & 0 & \dots & L'  & L\\
 0 & 0 & \dots & 0  & L\\
 \end{matrix}
 \right)
 $$

We are ready to state the conclusion of all this:

\begin{theorem}\label{th:generalU} Let $\F$ be an acceptable space of analytic functions $\U\To \Sigma$. For every $n\geq 2$ there exists an acceptable space $\mathscr T$ of analytic functions $\U \To \Sigma^n$ with the following properties:
\begin{itemize}
\item For every $f\in\F$, the array $\tau_{(n,0]}f:\U\To \Sigma^n$ belongs to $\mathscr T$, and $\|\tau_{(n,0]}f\|_{\mathscr T}=\|f\|_{\F}$.
\item For every $u\in\U$ one has $\mathscr T_u=\Fn_u$, with equivalent norms.
\end{itemize}
\end{theorem}
\begin{proof}
Fix a conformal map $\psi:\mathbb D\To\U$ and let $\mathscr G=\psi^*[\F]$. Then $\mathscr G$ is acceptable on $\D$ and $\F_u=\mathscr G_z$, where $u=\psi(z)$. If  $\mathscr G^{(n)}$ is  the space provided by Proposition~\ref{prop:Fnaccep}, we have:
\begin{itemize}
\item $\mathscr G^{(n)}$ is an acceptable space of $\Sigma^n$-valued functions on the disc.
\item The analytic family induced to $\mathscr G^{(n)}$ is $\mathscr G^{(n)}_z$, up to equivalence of norms.
\item If $g\in\mathscr G$, then $\tau_{(n,0]}g$ belongs to $\mathscr G^{(n)}$, and $\|\tau_{(n,0]}g\|_{\mathscr G^{(n)}}=\|g\|_{\mathscr G}$.
\end{itemize}
Moreover, we know from Section~\ref{sec:chainrule} that there is an analytic mapping $L(n,\cdot):\U\To M[n]$, the space of $n\times n$ matrices with complex coefficients, such that, if $u=\psi(z), f\in\F, g=f\circ\psi$, then
$$
\tau_{(n,0]}g(z)=L(n,u)\big{(}(\tau_{(n,0]}f)(u)\big{)}.
$$
Each $L(n,u)$ is upper triangular and invertible and restricts to a surjective isometry between $\F^{(n)}_u$ and $\G^{(n)}_z$ and so to an isomorphism from $\F^{(n)}_u$ to $(\G^{(n)})_z$.
Now, we continue with this $n$ fixed, and define $M:\D\To M[n]$ by $M(z)= L(n,\psi(z))^{-1}$. Consider the space
$$
M_*[\G^{(n)}]=\left\{H\in\mathscr H(\D,\Sigma^n): H(w)=M(w)(G(w)), \text{ with } G\in\G^{(n)}\right\}.
$$
It should be obvious by now that $
M_*[\G^{(n)}]$ is an acceptable space on the disc and also that $(M_*[\G^{(n)}])_z=
\F^{(n)}_u$, with equivalent norms, where $u=\psi(z)$. Finally, set $\mathscr T=(\psi^{-1})[M_*[\G^{(n)}]]$ and check the details.
\end{proof}

There is a puzzling fact in that one is much less interested in which are the spaces $\mathscr T$ appearing in Theorem \ref{th:generalU} than  in their mere existence. Indeed, $\mathscr T$ has been constructed to provide a framework that legitimates the manipulations we will perform next. On the other hand, the formalism developed in this paper for acceptable spaces is rather satisfactory in the sense that produces, under minimal hypotheses, both the Rochberg spaces and the process to derive them. A reader interested in interpolation theory could miss some concrete applications beyond Section \ref{secondaxis}. The main obstacle to derive ``classical" interpolation results from the material in Sections~\ref{sec:disc} and \ref{sec:general} is that, while admissible interpolation families lead to admissible spaces of analytic functions in the way explained in Section~\ref{sec:int-fam}, we do not know how to travel the way back, if there is a way back. Precisely, assume that $\F$ is an admissible space on the disc and let us fix $0<r<1$. Under which conditions one can guarantee that the spaces $(\F_z)_{|z|=r}$ form an interpolation family so that a new admissible space $\mathscr X$ can be eventually formed? And, if so, do the new interpolation spaces $(\mathscr X_z)_{|z|<r}$ agree with the old ones $\F_z$?

\section{Derivation of Rochberg families}\label{sec:rambling}

Let $\F$ be an acceptable space on $\U$. Fix  $m\geq 2$ and let $\mathscr T$ be the space provided by Theorem~\ref{th:generalU}
so that $\mathscr T^{(1)}_z = \Fm_z$; the fact that $\mathscr T$ depends on the choice of a conformal map does not affect the ensuing considerations. Since $\mathscr T$ is acceptable, given any integer $n\geq 2$ one can construct the corresponding Rochberg spaces $\mathscr T^{(n)}_z$ and the associated exact sequences (\ref{poz}) they naturally form. This section makes the first steps in the study of these objects. While our knowledge on this issue is very limited, the general impression is that one arrives to certain degenerate versions of the Rochberg spaces generated by the original $\F$.

Let us agree on the following notations.
For fixed  $m\geq 1$, if $\mathscr T$ is the space provided by Theorem~\ref{th:generalU}
so that $\mathscr T_z= \Fm_z$ for all $z\in\mathbb U$. Let us fix $z\in \mathbb U$ for the remainder of the section, write $F[m,n] = \mathscr T_z^{(n)}$ and rename the exact sequences entwinning the successive Rochberg spaces of $\mathscr T$ as
\begin{equation}\label{eq:mnk}\begin{CD}
0@>>> F[m,n] @>\imath_{n,n+k}^m>> F[m,n+k]@> \pi_{n+k,k}^m >> F[m,k]@>>> 0
\end{CD}
\end{equation}

We describe the elements of $F[m,n]$ by means of $m\times n$-matrices with entries in the ambient space $\Sigma$ as follows. Each function in $\mathscr T$ can be written as $F=(f_{m-1},\dots,f_0)$ where $f_j:\U\To \Sigma$ are certain analytic functions. Thus, a typical element of $F[m,n]$ arises by evaluation of the following array of functions
 $$
 \left( \begin{matrix}
\frac{F^{(n-1)}}{(n-1)!} \\
 \frac{F^{(n-2)}}{(n-2)!}\\
 \vdots \\
 F'\\
F\\
 \end{matrix}
 \right)
 =
 \left( \begin{matrix}
\frac{f_{m-1}^{(n-1)}}{(n-1)!} & \frac{f_{m-2}^{(n-1)}}{(n-1)!} & \dots  &  \frac{f_1^{(n-1)}}{(n-1)!} &  \frac{f_0^{(n-1)}}{(n-1)!}\\

\frac{f_{m-1}^{(n-2)}}{(n-2)!}  &\frac{f_{m-2}^{(n-2)}}{(n-2)!}& \dots & \frac{f_{1}^{(n-2)}}{(n-2)!}  & \frac{f_0^{(n-2)}}{(n-2)!}\\
 \vdots & \vdots & \ddots &\vdots &\vdots\\
  f'_{m-1} &   f'_{m-2}  & \dots & f_1'  & f_0'\\
 f_{m-1} &   f_{m-2}  & \dots & f_1  & f_0\\
 \end{matrix}
 \right)
 $$
at $z$. There is a quite natural operator $E_{m,n}:\F^{(m+n-1)}_z\To F[m,n]$. To see which one is, pick $x=(x_{m+n-2}, \dots, x_0)$ in $\F^{m+n-1}_z$. Let $f\in\F$ be an extremal for $x$ so that $x=\tau_{(m+n-1,0]}f(z)$ and put $F(\cdot)=\tau_{(m,0]}f(\cdot)$. Then $F\in\mathscr T$ and (the transpose of) $\tau_{(n,0]}F$ is
$$
 \left( \begin{matrix}
\frac{F^{(n-1)}}{(n-1)!} \\
 \frac{F^{(n-2)}}{(n-2)!}\\
 \vdots \\
 F'\\
F\\
 \end{matrix}
 \right)
 =
 \left( \begin{matrix}
\frac{f^{(m-1+n-1)}}{(m-1)!(n-1)!} & \frac{f^{(m-2+n-1)}}{(m-2)!(n-1)!} & \dots  &  \frac{f_1^{(n-1)}}{(n-1)!} &  \frac{f^{(n-1)}}{(n-1)!}\\

\frac{f^{(m-1+n-2)}}{(m-1)!(n-2)!}  &\frac{f^{(m-2+n-2)}}{(m-2)!(n-2)!}& \dots & \frac{f_{1}^{(n-2)}}{(n-2)!}  & \frac{f^{(n-2)}}{(n-2)!}\\
 \vdots & \vdots & \ddots &\vdots &\vdots\\
 \frac{1}{(m-1)!}f^{(m)}  &   \frac{1}{(m-2)!}f^{(m-1)}  & \dots & f''  & f'\\
\frac{1}{(m-1)!}f^{(m-1)}  &   \frac{1}{(m-2)!}f^{(m-2)}  & \dots & f'  & f\\
 \end{matrix}
 \right)
$$
Evaluating at $z$ we obtain
$$
E_{m,n}(x)=\left( \begin{matrix}
\frac{(m-1+n-1)!}{(m-1)!(n-1)!}x_{m-1+n-1} & \frac{(m-2+n-1)!}{(m-2)!(n-1)!}x_{m-2+n-1} & \dots  &  n x_n &  x_{n-1}\\
\frac{(m-1+n-2)!}{(m-1)!(n-2)!} x_{m-1+n-2}  &\frac{(m-2+n-2)!}{(m-2)!(n-2)!}x_{m-2+n-2}& \dots &(n-1)x_{n-1} & x_{n-2}\\
 \vdots & \vdots & \ddots &\vdots &\vdots\\
mx_m  &  (m-1)x_{m-1}  & \dots & 2x_2 & x_1\\
 x_{m-1}   &   x_{m-2}  & \dots & x_1  & x_0\\
 \end{matrix}
 \right)=\left( \frac{(i+j)!}{i!j!}x_{i+j}\right)_{n > i \ge 0, m> j \ge 0}
$$
It is clear that each $E_{m,n}$ is injective and continuous.
We shall see very soon that $E_{m,n}$ is an embedding with complemented range if $m$ or $n$ is $2$. To this end we need the following remark that implicitly concerns the pushout construction. We apologize for the tendentious notation.

\begin{lemma}\label{lem:KYZ}
Let $Z$ be a Banach space and let $K$ and $Y$ be closed subspaces of $Z$, with $K\subset Y$. Assume one has another Banach space $\PO$ and a commutative diagram
$$
\xymatrixcolsep{3pc}
\xymatrix{
0\ar[r] & K  \ar[r]^{\text{\rm inclusion}} \ar[d]_{\text{\rm inclusion}} & Z \ar[r]^{\text{\rm quotient}}  \ar[d]^E & Z/K \ar[r] \ar@{=}[d] & 0\\
0\ar[r] & Y  \ar[r]^{J} & \PO \ar[r]^Q & Z/K \ar[r] & 0
}
$$
with exact rows.
Then $E$ is an embedding with complemented range and $\PO/E[Z]$ is isomorphic to $Y/K$. In particular $\PO$ is isomorphic to $Y/K\oplus Z$.
\end{lemma}

\begin{proof} The three-lemma tell us that $E$ is an embedding and after a short reflection on the meaning of the operator $Q$ one realizes that $\PO=J[Y]+E[Z]$, so that $(y,z)\longmapsto J(y)+E(z)$ is open from $Y\oplus Z$ onto $\PO$. Define $U:Y/K\oplus Z\To \PO$ letting $U(y+K,z)=J(y)+E(z-y)$, which is an operator whose inverse can be obtained as follows: given $x\in\PO$ take $y\in Y$ and $z\in Z$ such that $x=J(y)+E(z)$ and set $V(x)=(y+K, z+y)$. Check, check, check.
\end{proof}

The copies of $Z$ and $Y/K$ inside $\PO$ that arise by restricting $U$ to each ``factor'' are obvious: the restriction of $U$ to $Z$ is just $E$; as for $Y/K$ one has
\begin{equation}\label{eq:U=J-E}
U(y+K,0)=J(y)-E(y)
\end{equation}
which depends only on the class of $y$ in $Y/K$ since $J$ and $E$ agree on $K$.

\begin{proposition}
For each $m\geq 1$ and $z\in\mathbb U$ the operator $E_{m,2}:\F^{(m+1)}_z\To F[m,2]$ is an embedding with complemented range and the quotient of $F[m,2]$ by $E_{m,2}[\F^{(m+1)}_z]$ is isomorphic to $\F^{(m-1)}_z$.
\end{proposition}

\begin{proof}
We consider $\F_z$ and $\F^{(m)}_z$ as subspaces of $\F^{(m+1)}_z$ and check that the following diagram is commutative
\begin{equation}\label{dia:Fm2}
\xymatrixcolsep{3pc}
\xymatrix{
0\ar[r] & \F_z\ar[r]^{\imath_{1,m+1}} \ar[d]^{\imath_{1,m}}& \F^{(m+1)}_z\ar[r] \ar[d]^{E_{m,2}} & \F^{(m)}_z \ar[r] & 0\\
0\ar[r] & \F^{(m)}_z\ar[r]^-{m\, \imath^m_{1,2}} & F[m,2]\ar[r]&\F^{(m)}_z\ar@{=}[u] \ar[r] & 0 }
\end{equation}
where we have identified $ \F^{(m)}_z$ with $F[m,1]$ in the obvious way.
Given $x=(x_m,x_{m-1},\dots,x_0)$ in $\F^{(m+1)}_z$ one has
\begin{equation}\label{eq:Em2}
E_{m,2}(x_m,x_{m-1},\dots,x_0)=
\left( \begin{matrix}
mx_m  &  (m-1)x_{m-1}  & \dots & 2x_2 & x_1\\
 x_{m-1}   &   x_{m-2}  & \dots & x_1  & x_0\\
 \end{matrix}
 \right)
\end{equation}
The left square is commutative since for $x\in \F_z$ the two possible compositions lead to
$$
\left( \begin{array}{cccc}
mx  &  0 & \dots & 0\\
0  &   0  & \dots & 0
 \end{array} \right)
$$
The right square is commutative as well: given $(x_m,x_{m-1},\dots,x_0)$ in $\F^{(m+1)}_z$ one has
$$
\pi^2_{2,1} E_{m,2}(x_m,x_{m-1},\dots,x_0) =
\pi^2_{2,1} \left( \begin{matrix}
mx_m  &  (m-1)x_{m-1}  & \dots & 2x_2 & x_1\\
 x_{m-1}   &   x_{m-2}  & \dots & x_1  & x_0\\
 \end{matrix}
 \right) = (x_{m-1},\dots,x_0).
$$
Applying the preceding lemma concludes the proof.
\end{proof}
The remark after the lemma shows where the copies of $\F^{(m+1)}_z$ and $\F^{(m-1)}_z$ are located in $F[m,2]$. The first one is given by the action of $E_{m,2}$, described by (\ref{eq:Em2}). The position of the complementary copy of $\F^{(m-1)}_z$ is defined by (\ref{eq:U=J-E}):  if $(y_{m-2},\dots, y_0)\in \F^{(m-1)}_z$ and $\tilde{y}=(y_{m-1},y_{m-2},\dots, y_0)$ is a ``lifting'' in   $\F^{(m)}_z$ we have
$$
U(\tilde{y}+\F_z,0)=m\imath^m_{1,2}\tilde{y}- E_{m,2}(\underbrace{y_{m-1},y_{m-2},\dots, y_0,0}_{\imath_{m,m+1}(\tilde{y})}) =
\left( \begin{matrix}
0 &  y_{m-2} & 2y_{m-3}  & \dots & (m-2)y_1 &  (m-1)y_0\\
-y_{m-2} & -y_{m-3}  & -y_{m-4}& \dots &    -y_0 & 0\\
 \end{matrix}
 \right)
$$
Reversing the parameters leads to similar conclusions:

\begin{proposition}
For each $m\geq 1$ and $z\in\mathbb U$ the operator $E_{2,m}:\F^{(m+1)}_z\To F[2,m]$ has complemented range and the quotient of $F[2,m]$ by  $E_{2,m}[\F^{(m+1)}_z]$ is isomorphic to $\F^{(m-1)}_z$.
\end{proposition}

\begin{proof}
We write the proof when $\mathbb U=\D$ with base point at the centre of the disc. The general case follows suit. Let us check that $F[2,m]$ fits into a commutative diagram
\begin{equation}\label{dia:F2m}
\xymatrixcolsep{3pc}
\xymatrix{
0\ar[r] & \F_0\ar[r] \ar[d]& \F^{(m+1)}_0\ar[r] \ar[d]^{E_{2,m}} & \F^{(m)} _0\ar[r] & 0\\
0\ar[r] & \F^{(m)}_0\ar[r]^-{mJ} & \F[m,2]\ar[r]^Q &\F^{(m)}_0\ar@{=}[u] \ar[r] & 0 }
\end{equation}
with exact rows. Recall what we agreed on unlabelled arrows. The other operators are defined as follows:
$$
E_{2,m}(x_{m},\dots,x_0)=\left( \begin{matrix}
mx_m  &  x_{m-1} \\
 (m-1) x_{m-1}   &   x_{m-2} \\
 \dots &\dots  \\
 2x_2 & x_1\\
  x_1  & x_0\\
 \end{matrix}
 \right);\, J(y_{m-1},\dots,y_0)=\left( \begin{matrix}
y_{m-1}  &  0\\
 y_{m-2}   &   0 \\
 \dots &\dots  \\
 y_1 & 0\\
  y_0  & 0\\
 \end{matrix}
 \right);\, Q\left( \begin{matrix}
u_{m-1}  &  v_{m-1}\\
u_{m-2}   &   v_{m-2} \\
 \dots &\dots  \\
 u_1 & v_1\\
  u_0  & v_0\\
 \end{matrix}
 \right)= (v_{m-1},\dots,v_0)
$$
While it is clear that the diagram commutes (when one replaces each space by its containing $\Sigma^k$) the continuity of $J$ and $Q$ is not completely obvious.

But an analytic function $F:\mathbb D\To \Sigma^2$ belongs to $\F^{(2)}$ if and only if there are $f_0,f_1\in \F$ such that  $F=(f_0'+f_1, f_0)$ in which case $\|F\|_{\F^{(2)}}=\|f_0\|_{\F}+ \|f_1\|_{\F}$; take $n=2$ in Corollary~\ref{cor:and-also}.

This implies that if $f_0,f_1$ are in $\F$ and $f_0=\sum_{k\geq 0}a_k z^k$ and $ f_1=\sum_{k\geq 0}b_k z^k$ are their respective Taylor expansions, then
\begin{equation}\label{eq:M}
M=
\left( \begin{matrix}
ma_m+b_{m-1}  &  a_{m-1}\\
 (m-1)a_{m-1}+ b_{m-2}   &   a_{m-2}  \\
 \dots &\dots  \\
 2a_2+b_1 & a_1\\
  a_1+b_0  & a_0\\
 \end{matrix}
 \right)\text{ belongs to $F[2,m]$, with } \|M\|_{F[2,m]}\leq
\|f_0\|_{\F}+ \|f_1\|_{\F}.
\end{equation}
and that all points of $F[2,m]$ have that form. Hence $J$ is bounded (actually contractive) from $\F^{(m)}_0$ to $F[2,m]$: given $x\in \F^{(m)}_0$ take an extremal $f\in \F$ with $x=\tau_{(m,0]}f(0)$, set $F=(f',0)$ (that is, $f_0=0,f_1=f$) and evaluate $F$ at the origin. Since all elements of $F[2,m]$ can be written as in (\ref{eq:M}) we see that their right columns are in $\F^{(m)}_0$ and that $Q$ is onto, with $\|Q:F[2,m]\To\F^{(m)}_0\|\leq 1$.

Clearly $J$ is injective. It remains to check that $\ker Q$ agrees with the image of $J$. One containment is trivial since $QJ=0$. As for the other assume $M\in F[2,m]$ is such that $Q(M)=0$. If we write $M$ as in (\ref{eq:M}), with $f_0=\sum_{k\geq 0}a_k z^k$ and $ f_1=\sum_{k\geq 0}b_k z^k$ in $\F$ we have that $a_k=0$ for $0\le k <m$, so that
$$
M=
\left( \begin{matrix}
ma_m+b_{m-1}  &  0\\
 b_{m-2}   &   0 \\
 \dots &\dots  \\
 b_1 & 0\\
  b_0  & 0\\
 \end{matrix}
 \right)\qquad\text{with}\qquad  \begin{cases}
 f_0=a_mz^m+a_{m+1}z^{m+1}+\cdots\\
 f_1=b_0+b_1z+b_{2}z^{2}+\cdots
 \end{cases}
$$
Since $f_0(0)=0$ we have that $f(z)=f_1(z)+mf_0(z)/z$ defines a function in $\F$; letting $y=\tau_{(m,0]}f(0)$ it should be obvious that $y\in \F^{(m)}_0$ is such that $J(y)=M$.
The proof concludes using Diagram (\ref{dia:F2m}) and the preceding lemma.
\end{proof}

The just proved proposition describes, in particular, the sucessive Rochberg (derived) spaces of the ``analytic family'' of the Kalton-Peck spaces; see Section~\ref{sec:couples}. It turns out that the spaces $F[m,2]$ and $F[2,m]$ are isomorphic since they are isomorphic to $\F^{(m+1)}_z\oplus   \F^{(m-1)}_z$.

It is both tempting and hasty to
conjecture that $E_{m,n}$ is always an embedding with complemented range with $F[m,n]/E_{m,n}[\F^{(m+n-1)}]$ isomorphic to  $F[m-1,n-1]$ for $m,n\geq 2$. We do not even know whether $E_{3,3}$ is an embedding or if $F[3,3]$ has a subspace isomorphic to $\F^{(5)}_z$.

\section{The solution of some problems. Counter-examples}\label{sec:appl}
In this section we will solve some problems left unanswered in \cite{cck,ccfg,Coifman1982,rochpac}.

\subsection{A totally incomparable family with nonsingular derivation at any point}\label{incomparable}

Recall that two Banach spaces are said to be totally incomparable if they do not admit isomorphic infinite dimensional subspaces. Recall also that an operator between Banach spaces is said to be strictly singular if its restrictions to infinite dimensional subspaces are never an isomorphism.

The paper \cite{ccfg2} is devoted to different aspects of the stability of the differential process associated to an analytic family $(\mathscr C_z)$. One problem not considered, though implicit, there is whether the total incomparability of the spaces $\mathscr C_t$ in a neighborhood of $\theta$ forces the quotient map $\pi_{2,1}: \mathscr C_z^{(2)}\To \mathscr C_z$ to be singular.

The answer is negative. Indeed, if $m\ge 2$, the quotient map $F[m,2]\To \Fm_z$ in Diagram (\ref{eq:Em2}) is {\em never} strictly singular because the composition
$$
\xymatrixcolsep{3.5pc}
\xymatrix{
\F_z \ar[r] & \F^{(m+1)}_z \ar[r]^{E_{m,2}} & F[m,2]\ar[r] & \F^{(m)}_z
}
$$
agrees with the natural inclusion $\imath_{1,m}$. It therefore suffices to consider a couple $(X_0, X_1)$ of Banach spaces and some $m\geq 2$ for which the spaces $\mathscr C(X_0,X_1)^{(m)}_t$  are mutually totally incomparable for $0<t<1$. This is easily achieved for all $m\geq 2$ taking $X_0=\ell_\infty, X_1=\ell_1$ since in this case, the spaces $\mathscr C(X_0,X_1)^{(m)}_t$, begin ``iterated'' twisted sums of $\ell_p$ for $p=1/t$ are $\ell_p$-saturated, by a simple 3-space argument; cf.
\cite[Theorem~3.2.d]{castgonz}.



\subsection{Answer to a question of Rochberg}\label{answroch}

In the seminal paper \cite[p.~266, last paragraph of Section 6]{rochpac}, Rochberg observes that, when $\mathscr F$ is the Calder\'on space associated to a couple of Banach lattices with associated differential $\Omega$ then $\Omega^{1,k}(f)$ depends only on $f$ and $\Omega^{1,1}(f)$. He asked if the same is true for arbitrary families. The answer is strongly negative since one can build, for each $n\geq 2$, an admissible family such that $ \Omega^{1,k}_z=0$ for $1\leq k<n$ but $ \Omega^{1,k}_z$ is not trivial for $k\geq n$.

Let us proceed with the counter-example. Fix a function $\omega:\D\To \mathbb S$ that extends to an analytic function on a neighborhood of $\overline{\mathbb D}$ that we denote again $\omega$. We set $p(z)=1/\Re(\omega(z))$. Consider the function space $\mathscr Z[\omega]$
which consists of those continuous functions $F:\overline{\mathbb D}\To\ell_\infty$ which are analytic on $\mathbb D$ and such that
$ \|F\|=\sup_{|z|\leq 1}\|F(z)\|_{\ell_{p(z)}}<\infty$. One has:

\begin{lemma}$\;$
\begin{itemize}
\item[(a)] $\mathscr Z[\omega]$ is an admissible space.
\item[(b)] $\mathscr Z[\omega]_z=\ell_{p(z)}$ for every $z\in \mathbb D$.
\item[(c)] Given $|\zeta|<1$ and $f\geq 0 $ normalized in $ \ell_{p(\zeta)}$, the function $\overline{\mathbb D}\To\ell_\infty$ defined by $F(z)=f^{\omega(z)/\omega(\zeta)}$ is normalized in $\mathscr Z[\omega]$ and $F(\zeta)=f$.
\end{itemize}
\end{lemma}

\begin{proof}
(a) It is clear that for each $z\in\mathbb D$ the evaluation $\delta_z$ is bounded as a map $\mathscr Z[\omega] \To \ell_\infty$. Since conformal automorphisms of the open unit disc extend continuously to the boundary (they are M\"obius transformations) in order to stablish that $\mathscr Z[\omega]$ has the required invariance property, it suffices to check that for each $F\in \mathscr Z[\omega]$ one has
$$
\|F\|=\sup_{z\in\mathbb T}\|F(z)\|_{\ell_{p(z)}},
$$
which follows from the maximum principle. The space $\mathscr Z[\omega]$ is complete since a uniform limit of analytic functions is analytic. 
 Part (b) follows from the very definition of the norm of $\mathscr Z[\omega]$ and (c), which we prove next: Fix $\zeta\in \D$ and set $p_0=p(\zeta)$ and $\omega_0=\omega(\zeta)$. Pick then a nonnegative, normalized $f\in\ell_{p_0}$ and define $F: \overline{\mathbb D}\To\ell_\infty$ by
$$
F(z)=f^{\omega(z)/\omega_0}
$$
with the convention that each power of zero is again zero. It is clear that $F$ is continuous on the closed disc and analytic on the interior. We are thus done because $F\in \mathscr Z[\omega]$ since for every $z\in\overline{\mathbb D}$,
 $$
\big\|F(z)\big\|_{\ell_{p(z)}}=\big\| f^\frac{\omega(z)}{\omega_0} \big\|_{\ell_{p(z)}}
= \big\| f^\frac{\Re\omega(z)}{\omega_0} \big\|_{\ell_{p(z)}}
=\big\| f^{1/\omega_0} \big\|_{\ell_{1}}^{1/p(z)}
=\big\| f^{p_0} \big\|_{\ell_{1}}^{1/p(z)}= 1.\qedhere$$
 \end{proof}
The answer to Rochberg's question comes now. For each $z\in\D$, let $\Omega_z$ be the differential generated by $\mathscr Z[\omega]$ at $z$.

\begin{proposition}
If $\omega'$ has a zero of order $k\geq 1$ at $\zeta$, then
$\begin{cases}
    \Omega_\zeta^{n,m} = 0 & \text{for $n+m \leq k+1;$} \\
    \Omega_\zeta^{n,m} \nsim 0 & \text{for $n+m\geq k+2$.}
  \end{cases}
$
\end{proposition}


\begin{proof}
The hypothesis means that $\omega'(\zeta)=\dots=\omega^{(k)}(\zeta)=0, \omega^{(k+1)}(\zeta)\neq 0$ and for $|z-\zeta|$ small enough we have
$$
\frac{\omega(z)}{\omega_0}=1+\sum_{n=k+1}^\infty a_n(z-\zeta)^n
$$
with $a_{k+1}\neq 0$. Set $a(z)=\omega(z)/\omega_0-1$, so that $a(z)$ has a zero of order $k+1$ at $z=\zeta$, with $a^{(k+1)}(\zeta)=\omega^{(k+1)}(\zeta)/\omega_0 =(k+1)!a_{k+1}$.
Take a positive, normalized $f\in \mathscr Z[\omega]_\zeta=\ell_{p_0}$ and let $F$ be the extremal provided above:
$$
F(z)=  f^{\omega(z)/\omega_0} =\exp\left( \frac{\omega(z) \log f}{\omega_0}	\right)= \exp\big( (1+a(z)) \log f	\big)= f\,\exp\big(a(z)\log f\big).
$$
Differentiating $F$ we obtain $F'(\zeta)=\cdots=F^{(k)}(\zeta)=0$ which immediately implies that
$\Omega_\zeta^{n,m}$ is bounded for $n+m\leq k+1$; which, after induction on $k$, gives
$$\mathscr Z[\omega]_\zeta^{(k+1)} =  \mathscr Z[\omega]_\zeta\oplus\cdots\oplus  \mathscr Z[\omega]_\zeta  =  \ell_{p_0}^{k+1}.$$
On the other hand,
$$
F^{(k+1)}(\zeta)= \frac{\omega^{(k+1)}(\zeta)}{\omega_0} \,f\, \log f
$$and thus
$$
\Omega_\zeta^{1,k+1}(f)=\left(c f\log|f|, 0, \dots, 0 \right),
$$
for some $c\neq 0$ and all normalized $f$. This map cannot be trivial since projection onto the first factor (which is bounded) yields the genuine (nontrivial) Kalton-Peck map; and therefore $\Omega_\zeta^{n,m}$ cannot be trivial when $n+m\geq k+2$ since
$\Omega_\zeta^{1,k+1}= \pi_{k+2,k+1}\Omega_\zeta^{k+2, k+2}\imath_{1,k+2}$.\end{proof}

The most obvious examples where the preceding Proposition applies are obtained taking $\omega(z)={1\over 2}+ rz^k$, with $0<r< {1\over 2} $ and $k\geq 2$.
In this case $\omega'(z)=krz^{k-1}$ has a zero of order $k-1$ at $0$ and thus
$\mathscr Z[\omega]_0^{(k)} \simeq (\ell_{1/2})^{k}= \ell_2$, while
$\mathscr Z[\omega]_0^{(k+1)}\simeq \mathscr Z_2 \oplus \ell_2^{k-1}\simeq \mathscr Z_2 \oplus \ell_2$ where
$\mathscr Z_2$ is the Kalton-Peck $Z_2$ space according to the notation in Section \ref{sec:corner}. The distribution of the spaces on $\mathbb T$ induced by the configuration $\omega$ consists of a ``periodic'' family of $\ell_{p(\theta)}$ spaces where $\theta\in [0,2\pi)$, and
$$
p(\theta)=\frac{1}{\Re({1\over 2}+re^{ik\theta})}=\frac{2}{1+2r\cos(k\theta)}.
$$
In \cite{cck}, it is shown that if the first differential $\Omega^{1,1}$ induced by an admissible space $\F$ is not trivial at $z$ then all $\Omega^{n,m}$ are nontrivial at $z$. Problem 6.1 in \cite{cck} asks whether the  reciprocal is true. The preceding example shows that the answer is negative. 

\subsection{A remark on ``reiteration'' for higher order differentials}
The spaces $\mathscr Z[\omega]$ are ``toy-examples'' of a more general construction by Coifman, Cwikel, Rochberg, Sagher and Weiss whose first order version is studied in \cite{ccfg}.
The key result we need is the basic reiteration for families of  \cite[Theorem 5.1]{Coifman1982}: 
{\em Let $\alpha: \T\To [0,1]$ be a measurable function such that both its infimum and supremum are attained. Let $(X_0, X_1)$ be an interpolation couple of Banach spaces.
Then $\mathcal X=\big\{(X_0, X_1)_{\alpha(\omega)} : {\omega \in \T}\big\} $ is an admissible interpolation family (in the sense of Section~\ref{sec:analytic-family}) and, if $\F=\F(\mathcal X)$ denotes the corresponding admissible space, then $\F_z= (X_0, X_1)_{\alpha(z)}$, with equality of norms, where $\tilde{\alpha}(z) = \int_{\partial \T} \alpha(\omega) dP_z(\omega)$ is the harmonic extension to $\mathbb D$ provided by the Poisson kernel $P_z$.}

The crucial fact inside the proof of this theorem is that if
 $\beta$ is the harmonic conjugate of $\tilde{\alpha}$ (with $\beta(0)=0$, say) and $\psi=\tilde{\alpha}+ i\,\beta$ then, given $z\in\D$ and $x\in\F_z=(X_0, X_1)_{\alpha(z)}$ one can obtain an extremal in $\F$ just taking an extremal $f$ for $x$ in $\mathscr C(X_0,X_1)$ and letting $f\circ\psi$.

It follows that if $\Omega_\theta$ denote the differentials associated to $\mathscr C(X_0,X_1)$ for $0<\theta<1$, then the differentials associated to $\F$ are given by $\Phi_z=\psi'(z)\Omega_{\psi(z)}= \psi'(z)\Omega_{\tilde{\alpha}(z)}$ at the first order level (this is \cite[Theorem 3.20]{ccfg}).

More generally, recall from Section~\ref{sec:chainrule} that for each $n$ there is an upper triangular matrix $\mathbf{FdB}[n,z,\psi]$ such that  $\tau_{(n,0]}(f\circ\psi)(z)= \mathbf{FdB}[n,z,\psi]  \tau_{(n,0]}f(\psi(z))$. It is clear that these matrices intertwine the successive differentials by the formul\ae
$$
\big(\Phi_z^{n,k} \big(\mathbf{FdB}[n,z,\psi]\,x\big), \mathbf{FdB}[n,z,\psi]\,x\big) = \mathbf{FdB}[n+k,z,\psi](\Omega_{\tilde{\alpha}(z)}^{n,k}(x), x\big),
$$
where $x\in \mathscr C(X_0, X_1)_{\alpha(z)}^{(n)}$, which matches with Diagram~\ref{dia:FuGz}. Note that $\Omega_{\tilde{\alpha}(z)}^{n,k}= \Omega_{\psi(z)}^{n,k}$.

\section{Appendix: A Fr\'echet algebra of analytic functions}

This Appendix contains the definition and basic properties of the algebra that supports the notion of an acceptable space. There are a number of reasons, most of them implicit in Section~\ref{sec:disc}, suggesting that one must start with an algebra of analytic functions on the disc which contains $\Aut(\mathbb D)$, the conformal automorphisms of the disc, and admits differentiation. The heuristic argumentation could be like this: Pick an admissible space $\F$. To generate $\F^{(2)}$ one would itch to set the space of functions $\{(f', f): f\in \F\}$; since $\F$ is admissible the product $\varphi f$ is in $\F$ for every $f \in \F$ and every conformal $\varphi$ as in Definition \ref{def:ad}. Now the point is that $((\varphi f)', \varphi f)$ does not behave as expected; and this is because $(\varphi f)' = \varphi' f + \varphi f'$. The term $\varphi f'$ is harmless since $\F$ is admissible, but
$\varphi' f$ is not, unless we somehow have a product $A\times \F \to \F$ by an algebra containing all derivatives of conformal maps.\medskip

In the search for $A$, observe that Banach algebras tend to not admit differentiation. So, instead of struggling to get an artificial one it is perhaps a better move to give up and look into the realm of Fr\'echet algebras, the natural habitat of derivatives. This is what we will do. A sequence of complex numbers $(c_n)$ is said to be rapidly decreasing if, for every positive real $\alpha$, one has $|c_n|=O(n^{-\alpha})$. Let us denote by $(s) $ the Fr\'echet space of rapidly decreasing sequences in its natural topology generated by the system of norms $|(c_n)|_\alpha = \sup_{n\geq 0}{|c_n| \: n^\alpha}$ for $0<\alpha<\infty.$  Note that $(s) $ contains every geometric progression $(a^n)_{n\geq 0}$ with $a\in\mathbb D$.\medskip

Let $A^\infty$ denote the linear space of all analytic functions $f:\mathbb D\To \mathbb C$
whose Taylor coefficients at the origin belong to $(s)$, with the obvious Fr\'echet topology. The following facts about $A^\infty$ are not hard to check:
\begin{itemize}
\item $A^\infty$ is a unital Fr\'echet algebra with the pointwise product (which does not correspond to the coordinatewise product of sequences, but to their convolution).
\item Ordinary differentiation is a continuous, linear endomorphism on $A^\infty$.
\item $\Aut(\mathbb D)\subset A^\infty$.
\end{itemize}

To prove the third point,
recall that all conformal automorphisms of the disc are M\"obius transformations and so they have the form
\begin{equation*}
\varphi(z)=\lambda\:\frac{z-a}{\overline{a}z-1}\quad\quad(|\lambda|=1, |a|<1).
\end{equation*}

Assuming $\lambda=1$ we have
$$
\varphi(z)=\frac{a-z}{1-\overline{a}z}=(a-z)\sum_{n\geq 0} \overline{a}^nz^n= a+
\sum_{n\geq 1} \left(a\overline{a}^n+\overline{a}^{n-1}\right)z^n \qquad\implies \qquad \varphi\in A^\infty.
$$

A minor drawback of the definition of $A^\infty$ is that everything seems to depend on the behaviour of the functions at the origin. We now characterize those functions which are in $A^\infty$ by  means of their boundary values. First of all, note that since the Taylor coefficients of any $f\in A^\infty$ are absolutely summable, $f$ extends continuously to the closed disc and in particular, it belongs to the disc algebra $A$ and even to the positive Wiener algebra $W^+$ (see definiton below). Let us denote this extension again by $f$. If $f$ is any function defined on the closed disc, then $f_{\mathbb T}$ denotes the ``boundary values'', that is, the periodic function defined by
$$
f_\T(t)=f(e^{it})\quad\quad(t\in\mathbb R)
$$
for real $t$. We denote by $Dg$ the ordinary derivative of $g:\mathbb R\longrightarrow\mathbb C$ with respect to the real variable $t$:
$$
Dg(t)=\lim_{h\to 0}\frac{g(t+h)-g(t)}{h}
$$
provided that limit exists. Given a continuous $2\pi$-periodic function $g:\mathbb R\to\mathbb C$, the $n$-th Fourier coefficient of $g$ is
$$
c_n=c_n(g)=\frac{1}{2\pi}\int_0^{2\pi}g(t)e^{-int}dt\quad\quad(n\in\mathbb Z).
$$
Note that if $g$ corresponds to the boundary values of some function of the disc algebra, then $c_n(g)=0$ for each $n< 0$. If, moreover, $f\in A$, then, by Cauchy formul\ae,
$$
f^{(n)}(0)= \frac{n!}{2\pi i}\oint_\mathbb T \frac{f(z)}{z^{n+1}}dz=
\frac{n!}{2\pi i}\int_0^{2\pi} \frac{f(e^{it})}{e^{i(n+1)t}}de^{it}=
\frac{n!}{2\pi}\int_0^{2\pi} \frac{f(e^{it})}{e^{int}}d t,
$$
so the the $n$-th Taylor coefficient of $f$ at the origin agrees with the
the $n$-th Fourier coefficient of ${f_\T}$.\medskip

Differentiability properties of periodic functions are related to the decay of their Fourier coefficients; indeed, a continuous $2\pi$-periodic function $g:\mathbb R\longrightarrow\mathbb C$ is smooth (that is, it has derivatives of all orders) if and only if
the (bilateral) sequence of Fourier coefficients of $g$ belongs to $(s)$; see, for instance, \cite[Lemma~3]{khavin}. All this shows:

\begin{lemma}
An analytic function $f:\D\to\mathbb C$ belongs to $A^\infty$ if and only if it has a continuous extension to the boundary which is smooth on $\T$.\hfill$\square$
\end{lemma}

\begin{corollary}
If $\psi\in\Aut(\D)$, then $\psi^*$ is a (continuous) automorphism of $A^\infty$.
\end{corollary}

\begin{proof}
Here, $\psi^*(a)=a\circ\psi$. It suffices to prove that $\psi^*$ is correctly defined (that is, it maps $A^\infty$ to itself) since the closed graph theorem implies continuity and the inverse is given by $(\psi^{-1})^*$. But the restriction of $\psi$ is a smooth diffeomorphism of $\mathbb T$ and so the boundary values of $\psi^*(a)$ are smooth if and only if so are those of $a$.
\end{proof}

We now graft our algebra $A^\infty$ into an arbitrary domain $\mathbb U$, conformally equivalent to the disc. Suppose $\psi:\mathbb U\longrightarrow\mathbb C$ is a conformal equivalence. We define
$$\psi^*[A^\infty]=
\{g:\mathbb U\to \mathbb C \text{ such that } g=f\circ\psi \text{ for some } f\in A^\infty\},
$$
with the obvious (Fr\'echet) topology. One has.

\begin{lemma}
$\psi^*[A^\infty]$ is independent of $\psi$.
\end{lemma}

\begin{proof}
Suppose $\psi_i: \mathbb U\longrightarrow\mathbb C$ are  conformal equivalences for $i=1,2$. Then $\psi=\psi_2\circ\psi_1^{-1}$ is an automorphism of the disc and so $\psi^*$ is an automorphism of $A^\infty$. It is unnecessary to continue.
\end{proof}

From now on we write $A^\infty_\U$ instead of $\psi^*[A^\infty]$. Of course $A^\infty_\D$ is just $A^\infty$.\medskip

The positive Wiener algebra $W^+$ is the algebra of all analytic functions on the disc whose Taylor coefficients at the origin are absolutely summable. It is clear that each function in $W^+$ has a continuous extension to $\overline{\D}$ and, in particular, it is bounded on $\D$. Given $f\in W^+$ we put $\|f\|_{W^+}=\sum_{n\geq 0}|c_n|$, where $f(z)=\sum_{n\geq 0} c_nz^n$ for $z\in \D$. As before, if $\psi:\mathbb U\longrightarrow\mathbb D$ is a conformal map,  we define
$$\psi^*[W^+]=
\{g:\mathbb U\to \mathbb C \text{ such that } g=f\circ\psi \text{ for some } f\in W^+\}
$$
and we transfer the norm of $W^+$ to $\psi^*[W^+]$ by stipulating that $\|g\|_{\psi^*[W^+]}= \|f\|_{W^+}$ provided $g=f\circ\psi$.

Note that $g:\U\To\mathbb C$ belongs to $\psi^*[W^+]$ if and only if there is $(c_n)_{n\geq 0}$ in $\ell_1$ such that $g(u)=\sum_{n\geq 0}c_n\psi(u)^n$ for all $u\in \U$ in which case $\|g\|_{\psi^*[W^+]}=\|(c_n)\|_{\ell_1}=\sum_n|c_n|$. One has:

\begin{lemma}\label{lem:AW}
$\psi^*[W^+]$ contains $A^\infty_\U$, and the inclusion is continuous.
\end{lemma}

\begin{proof}
Since $A^\infty_\U= \psi^*[A^\infty]$ it suffices to check that $W^+$ contains $A^\infty$ and the inclusion is continuous.
Which is obvious: every rapidly decreasing sequence $(c_n)_{n\geq 1}$ is absolutely summable, with $\|(c_n)\|_{\ell_1}\leq \frac{\pi}{6}|(c_n)|_2$.
\end{proof}

In spite of our good intentions, and rather unexpectedly, the grafted algebras $\AU$ are not closed under differentiation, even for very natural choices of $\U$. To convince the skeptical reader let us work out the following example: the function $\varphi(z)=(e^z-1)/(e^z+1)$ maps conformally the (horizontal) strip $\U=\{z: \Im(z)\in(-\frac{\pi}{2}, \frac{\pi}{2}) \}$ onto $\D$. Obviously $\varphi\in A^\infty_\U$. But if we write $w=e^z$, then
$$
\varphi'(z)=\frac{2w}{(w+1)^2}
$$
and we see that $\varphi'(z)$ has poles at $z=\pm  \frac{\pi}{2}i$. In particular
$\varphi'$ is unbounded on $\mathbb U$, and therefore it cannot be in $\AU$ which contains bounded functions only. In the end, this is one of the reasons why the generation of Rochberg families in general domains as in Section \ref{sec:general}
requires to move back and forth from $\mathbb U$ to $\D$ which, in turn, requires the Chain and Leibniz's rule.




\end{document}